\documentclass[a4paper,11pt]{amsart}
\usepackage[english]{babel}
\usepackage[utf8]{inputenc}
\usepackage[T1]{fontenc}

\usepackage{amssymb}
\usepackage{mathrsfs}
\usepackage{hyperref}
\usepackage[usenames,dvipsnames]{xcolor}
\usepackage{enumitem}	
\usepackage[pdftex]{graphicx}	

\newcommand{\bb}[1]{\mathbb{#1}}

\newcommand{\N}{\mathbb{N}}	
\newcommand{\R}{\mathbb{R}}	

\newcommand{\dd}{\,\mathrm{d}}	
\newcommand{\de}{\partial}		

\newcommand{\grad}{\nabla}
\newcommand{\Lip}{\mathrm{Lip}}
\newcommand{\card}{\operatorname{Card}}
\newcommand{\HH}{\bb H}

\newcommand{\interior}{\operatorname{int}}
\newcommand{\diam}{\operatorname{diam}}

\newcommand{\sgn}{\operatorname{sgn}}

\numberwithin{equation}{section}

\theoremstyle{plain}
\newtheorem{proposition}[equation]{Proposition}
\newtheorem{theorem}[equation]{Theorem}
\newtheorem{lemma}[equation]{Lemma}
\newtheorem{corollary}[equation]{Corollary}

\theoremstyle{definition}
\newtheorem{definition}[equation]{Definition}
\newtheorem{example}[equation]{Example}

\theoremstyle{remark}
\newtheorem{remark}[equation]{Remark}

\begin{document}

\title[BCP the Heisenberg group revisited]{The Besicovitch covering property in the Heisenberg group revisited}

\author[Sebastiano Golo]{Sebastiano Golo}

\address[Golo]{School of Mathematics, University of Birmingham, United Kingdom, Orcid ID: https://orcid.org/0000-0002-3773-6471}

\email{sebastiano.golo@fastmail.fm}

\thanks{S.G.~has been supported by the European Unions Seventh Framework Programme, Marie Curie Actions-Initial Training Network, under grant agreement n. 607643, ``Metric Analysis For Emergent Technologies (MAnET)''}

\author{S\'everine Rigot}

\address[Rigot]{Universit\'e C\^ote d'Azur, CNRS, LJAD, France}

\email{rigot@unice.fr}

\thanks{S.R.~is partially supported by ANR grant ANR-15-CE40-0018.}

\begin{abstract}
The Besicovitch covering property (BCP) is known to be one of the fundamental tools in measure theory, and more generally, a useful property for numerous purposes in analysis and geometry. We prove both sufficient and necessary criteria for the validity of BCP in the first Heisenberg group equipped with a homogeneous distance. Beyond recovering all previously known results about the validity or non validity of BCP in this setting, we get simple descriptions of new large classes of homogeneous distances satisfying BCP. We also obtain a full characterization of rotationally invariant distances for which BCP holds in the first Heisenberg group under mild regularity assumptions about their unit sphere.
\end{abstract}


\maketitle

\setcounter{tocdepth}{2}
\phantomsection
\tableofcontents



\section{Introduction}

The theory of differentiation of measures originates from works of Besicovitch in the Euclidean space (\cite{BesicovitchI45},~\cite{BesicovitchII46}). These pioneering works, as well as subsequent developments of the theory in the metric setting, rely as fundamental tools on suitable covering properties of the ambient space. Among such covering properties, the Besicovitch covering property (BCP), see Definition~\ref{def:bcp}, plays a central role for the validity of the differentiation theorem for every locally finite Borel regular measure over a separable metric space, see for instance~\cite{Preiss83} or the survey~\cite{Rigot_differentiation_measures}. We also stress that, beyond its role in the theory of differentiation of measures, BCP is a covering property that is geometric in nature and can be used as a useful tool to deduce global properties on the ambient space from local ones.

\medskip

We focus in the present paper on the first Heisenberg group $\HH$ with the aim of describing new homogeneous distances satisfying BCP. Our results come in the continuity of~\cite{LeDonne_Rigot_Heisenberg_BCP} where the existence of homogeneous distances for which BCP holds on $\HH$ was first proved. We prove in the current paper sufficient conditions, see Theorem~\ref{thm:intro-sufficient-cdt} and Section~\ref{sect:sufficient-cdt}, and necessary conditions, see Theorem~\ref{thm:intro-necessary-cdt} and Section~\ref{sect:necessary-conditions}, for the validity of BCP on the first Heisenberg group equipped with a homogeneous distance. Beyond recovering the class of homogeneous distances for which it was proved in~\cite{LeDonne_Rigot_Heisenberg_BCP} that BCP holds, we will get as a consequence of our sufficient conditions rather simple descriptions of new large classes of homogeneous distances on $\HH$ for which BCP holds, see Section~\ref{subsect:applications}. We stress that the proof of Theorem~\ref{thm:intro-sufficient-cdt} requires original arguments that are independent from those used in~\cite{LeDonne_Rigot_graded_groups_BCP} or~\cite{LeDonne_Rigot_Heisenberg_BCP}. As a consequence of Theorem~\ref{thm:intro-sufficient-cdt} and Theorem~\ref{thm:intro-necessary-cdt}, we will also get a simple characterization of rotationally invariant homogeneous distances on $\HH$ for which BCP holds, under mild additional assumptions about the regularity of their unit sphere, see Theorem~\ref{thm:intro-characterization-rot-inv-dist}. One can for instance easily recover from these results the fact, first proved in~\cite{MR2103544}, that the Carnot-Carath\'eodory distance on $\HH$ does not satisfy BCP. One also recovers the fact that the Kor\'anyi distance do not satisfy BCP, this was already noticed in~\cite{Koranyi-Reimann95} and~\cite{Sawyer_Wheeden}. More importantly, one gets explicit homogeneous distances that are as close as one wants to the Kor\'anyi distance and that do satisfy BCP, see~Theorem~\ref{thm:intro-application}.

\medskip

Before stating our main results, we recall the definition of the Besicovitch covering property in the metric setting. Here, balls in a metric space $(X,d)$ are assumed to be closed. Namely, a ball denotes a set of the form $B(p,r) := \{q\in X : d(p,q) \leq r\}$ for some $p\in X$ and some $0<r<\infty$.

\begin{definition}[Besicovitch family of balls] \label{def:Besicovitchfamily}
We say that a family $\mathcal{B}$ of balls in a metric space $(X,d)$ is a \textit{Besicovitch family of balls} if, first, for every ball $B\in \mathcal{B}$ with center $p_B$, we have $p_B\not \in B'$ for all $B'\in \mathcal{B}$, $B\not= B'$, and, second, $\bigcap_{B\in \mathcal{B}} B \not= \emptyset$. 
\end{definition}

\begin{definition}[Besicovitch covering property] \label{def:bcp}
We say that a metric space $(X,d)$ satisfies the \textit{Besicovitch Covering Property} (BCP) if there exists a constant $C\geq 1$ such that $\card \mathcal{B} \leq C$ for every Besicovitch family $\mathcal{B}$ of balls in $(X,d)$. 
\end{definition}

In this paper, we identify the first Heisenberg group $\HH$ with $\R^2 \times \R$ equipped with the group law given by
\begin{equation*}
(v,z)\cdot(v',z'):=(v+v',z+z'  +\omega(v,v')/2)
\end{equation*}
where $\omega(v,v'): = v_1 v'_2 - v_2 v'_1$ for $v=(v_1,v_2), v'=(v'_1,v'_2) \in\R^2$, and equipped with the family of dilations $(\delta_t)_{t\geq 0}$ defined by $\delta_t(v,z) := (tv,t^2z)$. We recall that a distance $d$ on $\HH$ is said to be homogeneous if it is left-invariant and one-homogeneous with respect to the family of dilations, see Section~\ref{sect:preliminaries}. 

\medskip

For the sake of completeness, we recall that the validity of BCP for a homogeneous distance $d$ on $\HH$ characterizes the validity of the differentiation theorem for every Radon measure over $(\HH,d)$. Theorem~\ref{thm:bcp-diff-meas-heis} below is indeed a corollary of~\cite[Theorem~1.5]{LeDonne_Rigot_graded_groups_BCP}. We refer to~\cite{Preiss83} or~\cite{Rigot_differentiation_measures} for more general statements in the metric setting.

\begin{theorem} [{\cite[Theorem~1.5]{LeDonne_Rigot_graded_groups_BCP}}] \label{thm:bcp-diff-meas-heis}
Let $d$ be a homogeneous distance on $\HH$. The differentiation theorem holds for every Radon measure $\lambda$ over $(\HH,d)$, that is,
\begin{equation*} 
\lim_{r\downarrow 0} \frac{1}{\lambda(B(x,r))} \int_{B(x,r)} f(y) \, d\lambda (y) = f(x) \;\; \text{ for } \lambda \text{-a.e. } x\in X
\end{equation*}
for every $f\in L^1_{loc} (\lambda)$, if and only if $(\HH,d)$ satisfies BCP.
\end{theorem}

We briefly describe now our main results. Given a homogeneous distance $d$ on $\HH$, we know from~\cite{LeDonne_Nicolussi_regularity_spheres} that its unit ball centered at the origin $B$ can be described as
$$B = \{(v,z) \in \HH:\, v\in K,\ -\phi(-v)\le z\le \phi(v) \}$$
for some compact set $K\subset \R^2$ and some function $\phi: K \rightarrow [0,+\infty)$ locally Lipschitz on $\interior(K)$, which we call the \textit{profile of the homogeneous distance} $d$. We refer to Section~\ref{sect:preliminaries} for more details about profiles of homogeneous distances.

\medskip

The sufficient and necessary conditions for the validity of BCP proved in the present paper will be mainly given in terms of suitable properties for the profile of homogeneous distances.

\medskip

Namely, we prove the following sufficient condition, Theorem~\ref{thm:intro-sufficient-cdt} below. Given $v, w \in \R^2 \setminus\{0\}$, we denote by $\measuredangle (v,w)$ the unoriented angle between $v$ and $w$, and by $\|\cdot\|$ the Euclidean norm.

\begin{theorem} \label{thm:intro-sufficient-cdt}
Let $d$ be a homogeneous distance on $\HH$ with profile $\phi: K \rightarrow [0,+\infty)$. Assume that the following two conditions hold. First, there are $ m>  0$ and $\alpha \in (0,\pi / 2)$ such that, for a.e.~$v\in \interior(K)\setminus\{0\}$, we have, for all $w\in \R^2\setminus\{0\}$, $$\measuredangle (v,w) \leq \alpha \Rightarrow \langle \nabla\phi(v) , w \rangle \leq -m \|v\| \|w\|.$$
Second, there are $\kappa>0$ and $ M >  0$ such that $\{v\in\R^2 : \|v\|\leq \kappa\} \subset \interior(K)$ and $$\|\nabla\phi(v)\| \leq M \|v\| $$ for a.e.~$v\in \R^2$ such that $\|v\|\leq \kappa$. Then BCP holds in $(\HH,d)$.
\end{theorem}

Theorem~\ref{thm:intro-sufficient-cdt} will be obtained through an application of an abstract sufficient condition proved in Section~\ref{sect:sufficient-cdt}, see Theorem~\ref{thm:abstract-sufficient-cdt} and its corollary, Corollary~\ref{cor:abstract-sufficient-cdt}.

\medskip

Conversely, we also prove several necessary conditions. For homogeneous distances, without any further assumption on the regularity of their profile, our main necessary conditions for the validity of BCP are gathered in Theorem~\ref{thm:intro-characterization-rot-inv-dist} below. We refer to Section~\ref{sect:necessary-conditions} for more detailed statements. We denote by $\bb S^1$ the Euclidean unit circle in $\mathbb{R}^2$.

\begin{theorem} \label{thm:intro-necessary-cdt}
Let $d$ be a homogeneous distance on $\HH$ with profile $\phi: K \rightarrow [0,+\infty)$. Assume that BCP holds in $(\HH,d)$. Then there is a countable closed set $W\subset\bb S^1$ such that $\bb S^1 \setminus W$ can be written as a countable union of relatively open subsets $O_j$, and, for each $j$, there are $m_j>0$ and $\alpha_j \in (0,\pi/2)$ such that, for all $v\in \interior(K) \setminus \{0\}$ such that $\phi$ is differentiable at $v$ and $ v/\|v\| \in O_j$, we have, for all $w\in\R^2\setminus\{0\}$, $$\measuredangle (v,w) \leq \alpha_j \Rightarrow \langle \grad\phi(v) , w \rangle  \le -  m_j \|v\|\|w\|.$$ Furthermore, $\phi(0) = \max \{\phi(v): v \in K\}$ and $\phi$ is differentiable at $0$ with $\nabla \phi(0) = 0$.
\end{theorem}

Combining Theorem~\ref{thm:intro-sufficient-cdt} and Theorem~\ref{thm:intro-necessary-cdt}, one gets a particularly simple characterization, Theorem~\ref{thm:intro-characterization-rot-inv-dist} below, for the validity of BCP for rotationally invariant homogeneous distances, under mild regularity assumptions on their profile. Namely, we say that a homogeneous distance $d$ on $\HH$ is rotationally invariant if rotations around the $z$-axis - in the Euclidean sense - are isometries in $(\HH,d)$. Equivalently, a homogeneous distance $d$ on $\HH$ with unit ball centered at the origin $B$ is rotationally invariant if and only if the intersection of $B$ with the $v$-plane is a Euclidean disc with radius $r_d$ for some $r_d >0$ and $B$ can be described as 
\begin{equation*}
B = \{(v,z) \in \HH:\, \|v\| \leq r_d,\ -\varphi(\|v\|)\le z\le \varphi(\|v\|) \}
\end{equation*}
for some even function $\varphi : [-r_d,r_d] \rightarrow [0,+\infty)$ locally Lipschitz on $(-r_d,r_d)$, which we call the \textit{radial profile of the rotationally invariant homogeneous distance} $d$.

\begin{theorem} \label{thm:intro-characterization-rot-inv-dist}
Let $d$ be a rotationally invariant homogeneous distance on $\HH$ with radial profile $\varphi: [-r_d,r_d] \rightarrow [0,+\infty)$. Assume that $|\varphi'(s) / s|$ is essentially bounded in a neighbourhood of $0$, which holds in particular if $\varphi \in C^1((-r_d,r_d))$ with $\varphi'$ is differentiable at $0$. Then BCP holds in $(\HH,d)$ if and only if there is $m>0$ such that $\varphi'(s) \leq -ms$ for a.e.~$s\in (0,r_d)$.
\end{theorem}

Theorem~\ref{thm:intro-application} below gives an application of this characterization. We recall that the Kor\'anyi distance $d_0$ is the homogeneous distance on $\HH$ given by 
\begin{equation} \label{e:def-koranyi-dist}
d_0(0,(v,z)) := \left(\|v\|^4+16z^2\right)^{1/4}.
\end{equation}
For $\epsilon>0$, we define $d_\epsilon$ as the homogeneous distance on $\HH$ given by 
\begin{equation} \label{e:d_epsilon}
d_\epsilon(0,(v,z)) :=
\left( \epsilon^2 \|v\|^2 + d_0(0,(v,z))^2 \right)^{1/2}.
\end{equation}
Note that these homogeneous distances are rotationally invariant. Note also that 
\begin{equation*} \label{e:biLip-dist}
d_0 \leq d_\epsilon \leq (1+\epsilon) d_0
\end{equation*}
for every $\epsilon > 0$.

\begin{theorem} \label{thm:intro-application}
BCP hold in $(\HH,d_\epsilon)$ for every $\epsilon >0$ whereas BCP does not hold in $(\HH,d_0)$.
\end{theorem}

As already said, it has already been noticed, independently and at the same time, in~\cite{Koranyi-Reimann95} and~\cite{Sawyer_Wheeden}, that BCP does not hold for the Kor\'anyi distance. On the other hand, it has been proved in~\cite{LeDonne_Rigot_Heisenberg_BCP} that BCP holds in $(\HH,d_1)$. As an application of Theorem~\ref{thm:intro-characterization-rot-inv-dist}, we can hence recover both results, but, more importantly, we get the validity of BCP for all homogeneous distances $d_\epsilon$ with $\epsilon >0$. This gives in particular  an example of a metric space $(\HH,d_0)$ that does not satisfy BCP but whose distance $d_0$ is as close as one wants from bi-Lipschitz equivalent distances for which BCP holds. 

\medskip

Going back to the necessary conditions and to homogeneous distances that are not assumed to be rotationally invariant, we will prove some improvements of Theorem~\ref{thm:intro-necessary-cdt} under further regularity assumptions on the profile of the homogeneous distance, see Section~\ref{subsect:additional-nec-cdt}. In particular, Theorem~\ref{thm:intro-more-nec-cdt} below is a consequence of the results proved in this last section of the paper.

\begin{theorem} \label{thm:intro-more-nec-cdt}
Let $d$ be a homogeneous distance on $\HH$ with profile $\phi: K \rightarrow [0,+\infty)$. Assume that BCP holds in $(\HH,d)$ and assume that $\phi \in C^2(\interior (K))$. Then $\langle \grad\phi(w) , w\rangle < 0$ for all $w\in \interior(K) \setminus \{0\}$, $\phi(w) < \phi(0)$ for all $w\in K \setminus \{0\}$, and the Hessian of $\phi$ at $0$, $H\phi(0)$, is semi-definite negative. Moreover, if $H\phi$ is assumed to be differentiable at $0$, then $H\phi(0)$ is definite negative.
\end{theorem}

To conclude this introduction, we stress that it seems to be a quite delicate issue to fill in the little gap between the sufficient and the necessary conditions proved in the present paper in order to get a characterization of non-rotationally invariant homogeneous distances satisfying BCP on $\HH$, even under the additional assumptions about the  regularity of their profile. However, our results should be strong enough to cover the study of the validity of BCP for most explicit homogeneous distances one may work with in $\HH$.

\medskip

The rest of the paper is organized as follows. In Section~\ref{sect:preliminaries}, we recall our conventions about the Heisenberg group $\HH$, and we refine the description proved in~\cite{LeDonne_Nicolussi_regularity_spheres} about unit balls centered at the origin for homogeneous distances on $\HH$. We also fix some conventions and notations for later use. Applications of Theorem~\ref{thm:intro-sufficient-cdt} are given in Section~\ref{subsect:applications}, see for instance Corollary~\ref{cor:sufficiention-rot-inv} and Corollary~\ref{cor:sufficiention-rot-inv-bis}, which give the "if" part in Theorem~\ref{thm:intro-characterization-rot-inv-dist}, and the application for the proof of the validity of BCP for the homogeneous distances $d_\epsilon$ for $\epsilon >0$. We prove in Section~\ref{subsect:abstract-sufficient-cdt} an abstract sufficient condition on which the proof of Theorem~\ref{thm:intro-sufficient-cdt} will be based. The proof of Theorem~\ref{thm:intro-sufficient-cdt} itself is given in Section~\ref{subsect:proof-main-sufficient}. Section~\ref{sect:necessary-conditions} is devoted to the proof of several necessary conditions for the validity of BCP. Our arguments to get these necessary conditions are based on Proposition~\ref{prop:necessary-condition-bcp} which is a consequence of a sufficient condition for the non-validity of BCP proved in~\cite{LeDonne_Rigot_Heisenberg_BCP}, see Section~\ref{subsect:preliminaries-nec-cdt}. The proof of Theorem~\ref{thm:intro-necessary-cdt} can be found in Section~\ref{subsect:nec-cdt-general-case}. We also refer to Theorem~\ref{thm:necessary-cdt-int(K)}, Theorem~\ref{thm:radially-decreasing-profile} and Theorem~\ref{thm:diff-at-0} for more detailed statements. The "only if" part in Theorem~\ref{thm:intro-characterization-rot-inv-dist} will be given by Corollary~\ref{cor:necessary-rot-inv} and, as one of its applications, we will recover the fact that BCP does not hold for the Kor\'anyi distance. In section~\ref{subsect:additional-nec-cdt}, we show several improvements of the necessary conditions proved in the previous section under additional regularity assumptions about the profile of homogeneous distances. Theorem~\ref{thm:intro-more-nec-cdt} follows in particular from these improvements.

\section{Preliminaries} \label{sect:preliminaries}

We recall from the introduction that we identify the first Heisenberg group $\HH$ with $\R^2\times \R$ equipped with the group law given by
\begin{equation*}
(v,z)\cdot(v',z'):=\left(v+v',z+z' + \frac{1}{2}\omega(v,v')\right)
\end{equation*}
where
\begin{equation*}
\omega(v,v') := v_1 v'_2 - v_2 v'_1
\end{equation*}
for $v=(v_1,v_2), v'=(v'_1,v'_2) \in\R^2$. The identity element coincides with the origin denoted by $0=(0,0)$ and, for $p\in\HH$, we have $p^{-1} = -p$.

\medskip

We equip $\HH$ with the family of dilations $(\delta_t)_{t\geq 0}$ given by
\begin{equation*}
\delta_t(v,z) := (tv,t^2z).
\end{equation*}
This family of dilations is a one parameter group of group automorphisms in $\HH$.

\begin{definition} [Homogeneous distances] \label{def:homogeneous-distances}
A distance $d$ on $\HH$ is said to be \textit{homogeneous} if it is left-invariant, that is, $d(p\cdot q,p\cdot q')  = d(q,q')$ for all $p, q, q' \in \HH$, and one-homogeneous with respect to the family of dilations $(\delta_t)_{t>0}$, that is, $d(\delta_t(p), \delta_t(q)) = t d(p,q)$ for all $t>0$ and all $p, q \in \HH$.
\end{definition}

We recall that a homogeneous distance induces the Euclidean topology on $\HH$, see~\cite[Proposition~2.26]{LeDonne_Rigot_graded_groups_BCP}. Given a homogeneous distance $d$ on $\HH$, we denote by $B:=\{p\in\HH :\, d(0,p)\leq 1\}$ its unit ball centered at the origin. We define $\pi:\HH \rightarrow \R^2$ by $\pi(v,z) := v$. 

\medskip

The following description of unit balls centered at the origin for homogeneous distances on $\HH$ will play a central role in the present paper.

\begin{proposition} \label{prop:LDNG} Let $d$ be a homogeneous distance on $\HH$ and let $K:=\pi(B)$. Then $K$ is a compact convex set such that $0\in \interior(K)$ and $K$ is symmetric, that is, $K=-K$, and there is a bounded function $\phi: K \rightarrow [0,+\infty)$ locally Lipschitz on $\interior(K)$ such that 
$$B = \{(v,z) \in \HH :\, v\in K,\ -\phi(-v)\le z\le \phi(v) \}~.$$ Furthermore,
\begin{equation} \label{e:positivity-profile}
\phi(v) >0 \quad \text{for all } v\in \interior(K)
\end{equation}
and  
\begin{equation} \label{e:radial-continuity-phi}
\lim_{t\rightarrow 1^{-}} \phi(tv) = \phi(v) \quad \text{for all } v\in \partial K.
\end{equation}
\end{proposition}

We call the function $\phi: K \rightarrow [0,+\infty)$ given by Proposition~\ref{prop:LDNG} the \textit{profile of the homogeneous distance} $d$. 

\medskip

The first part of Proposition~\ref{prop:LDNG} follows from~\cite[Proposition 6.2]{LeDonne_Nicolussi_regularity_spheres}. To prove~\eqref{e:positivity-profile}, we first note that, since $0 \in \interior(B)$, there is $z>0$ such that $(0,z) \in B$ and hence $\phi(0) \geq z >0$. Next, let $v \in \interior(K) \setminus \{0\}$. Assume with no loss of generality that $v=(x,0)$ with $x>0$. Since $v\in \interior (K)$, there is $h>0$ such that $v_h:=(x,h), v_{-h}:=(x,-h) \in  \interior(K)$, and hence $(v_h,0), (v_{-h},0) \in B$. By left-invariance and homogeneity of the distance, it follows that $\delta_t (v_{-h},0) \cdot \delta_{1-t} (v_h,0) = (x,(1-2t)h, t(1-t)hx) \in B$ for all $t\in [-1,1]$. For $t=1/2$, we get $(v,hx/4) \in B$, and hence $\phi(v) \geq hx/4 >0$. To prove~\eqref{e:radial-continuity-phi}, let $v\in \partial K$. Then $(v,\phi(v)) \in B$ hence $\delta_t(v,\phi(v)) = (tv,t^2\phi(v)) \in B$ for all $t\in [0,1]$ by homogeneity of the distance. This implies that $t^2\phi(v) \leq \phi(tv)$, and hence $\phi(v) \leq \liminf_{t\rightarrow 1^{-}} \phi(tv)$. Conversely, since $B$ is closed, $\phi$ is upper semicontinuous on $K$. In particular, it follows that, for  $v\in \partial K$, we have $\limsup_{t\rightarrow 1^{-}}\phi(tv) \leq \phi(v)$.

\medskip

Note that a compact convex and symmetric set $K\subset \R^2$ contains $0$ in its interior if and only if $K$ is non empty with $K=\overline{\interior(K)}$. Furthermore, note that Proposition~\ref{prop:LDNG} implies that
$$\interior(B) = \{ (v,z) \in \HH:\, v\in \interior(K), -\phi(-v)< z< \phi(v) \}$$
and 
\begin{equation*}
\begin{split}
\partial B = \{ & (v,z) \in \HH:\, v\in \interior(K),\ z= \phi(v) \text{ or } z= -\phi(-v) \} \\
&\, \cup \{ (v,z) \in \HH:\, v\in \partial K,\ -\phi(-v)\le z\le \phi(v)\}~. 
\end{split}
\end{equation*}

\medskip

Next, note that if $d$ is a rotationally invariant homogeneous distance with radial profile $\varphi:[-r_d,r_d] \rightarrow [0,+\infty)$ (see the paragraph before Theorem~\ref{thm:intro-characterization-rot-inv-dist}) then $K=\pi(B)$ is a closed Euclidean disc with radius $r_d$ centered at the origin and, if $\phi:K\rightarrow[0,+\infty)$ denotes its profile, we have $\phi(v) = \varphi(\|v\|)$ for every $v \in K$. Hence~\eqref{e:radial-continuity-phi} together with the fact that $\varphi \in \Lip_{loc}((-r_d,r_d))$ imply that the radial profile is continuous on the closed segment $[-r_d,r_d]$, and hence its profile $\phi$ is continous on the closed set $K$. 

\medskip

However, we stress that the profile $\phi: K \rightarrow [0,+\infty)$ of a non-rotationally invariant homogeneous distance $d$ on $\HH$, although continuous on $\interior(K)$ and satisfying~\eqref{e:radial-continuity-phi}, need not be continuous on $K$ as shows Example~\ref{ex:non-cont-profile} below. We first give a criterion that might have its own interest and gives a large class of examples of profiles of homogeneous distances on $\HH$.

\begin{proposition} \label{prop:concave-profile}
Let $K\subset \R^2$ be a compact convex and symmetric set containing $0$ in its interior. Let $\phi: K \rightarrow [\diam(K)^2/16,+\infty)$ and assume that $\phi$ is concave on $\interior (K)$ and upper semicontinuous on $K$. Then $\phi$ is the profile of a homogeneous distance on $\HH$. Namely, setting 
$$B:=\{(v,z)\in\HH :\, v \in K, \, -\phi(-v) \leq z \leq \phi(v)\},$$ then
$$d(p,q) :=\inf \{t>0:\, \delta_{1/t} (p^{-1}\cdot q) \in B\}$$ defines a homogeneous distance on $\HH$ whose unit ball centered at the origin is the set $B$.
\end{proposition}

\begin{proof} 
To prove that such a function is the profile of a homogeneous distance on $\HH$, we need to prove that the set $B$ is a compact symmetric set that contains $0$ in its interior and that
\begin{equation} \label{e:triangular-inequality}
\delta_t (p) \cdot \delta_{1-t} (q) \in B
\end{equation}
for all $p, q \in B$ and all $t\in [0,1]$ (see~\cite[Example~2.37]{LeDonne_Rigot_graded_groups_BCP}). By positivity and upper semicontinuity of $\phi$ on the compact set $K$, the set $B$ is bounded and closed, hence compact. The facts that $B$ is symmetric, that is, $p^{-1}=-p \in B$ for every $p\in B$, and contains $0$ in its interior are straightforward. 

To prove~\eqref{e:triangular-inequality}, we note that $B$ is the closure of the set $$\{(v,z)\in\HH :\, v \in \interior(K), \, -\phi(-v) \leq z \leq \phi(v)\} = B_+ \cap B_-$$
where 
\begin{align*}
	B_+ &:=  \{(v,z) \in \HH:\, v \in \interior(K), \,  z\le \phi(v) \} ,\\
	B_- &:=  \{(v,z): v \in \interior(K),\,  -\phi(-v)\le z \} .
	\end{align*}
Hence it is sufficient to prove that $\delta_t (p) \cdot \delta_{1-t} (q) \in B_+$ for every $p, q \in B_+$ and $t\in [0,1]$, together with the similar statement replacing $B_+$ by $B_-$. We consider the case $p, q \in B_+$, the other case being similar. Let $p=(v,z)\in B_+$ and $q=(\bar v,\bar z)\in B_+$.
	Then
	\[
\delta_t (p) \cdot \delta_{1-t} (q)	
	= \left( tv + (1-t)\bar v , t^2z + (1-t)^2\bar z + \frac12 t(1-t) \omega(v,\bar v) \right) .
	\]
	Next, by definition of $B_+$ and since $\|v\|,\|v'\| \leq \diam(K)/2$, we have
	\begin{align*}
	t^2z + (1-t)^2\bar z + \frac12 t(1&-t) \omega(v,\bar v) \\
	&\le t^2 \phi(v) + (1-t)^2\phi(\bar v) + \frac18 t(1-t) \diam(K)^2 .
	\end{align*}
	Now, we conclude the proof with the following observation. For $A,B\ge 1/4$, we have 
	\begin{equation*}
	t^2A+(1-t)^2B+\frac12 t(1-t) \le tA+(1-t)B 
	\end{equation*}
	for every $t\in [0,1]$. Indeed, since $A,B\ge 1/4$, then $1/2 \le 2 \sqrt{AB}$.
	Therefore, 
	\begin{align*}
	t^2A+(1-t)^2B+\frac12 t(1-t)
	&\le t^2A+(1-t)^2B+ 2 t(1-t) \sqrt{AB} \\
	&= (t\sqrt{A}+(1-t)\sqrt{B})^2.
	\end{align*}
	Since the function $g:[0,1] \rightarrow\R$ given by $g(t):= (t\sqrt{A}+(1-t)\sqrt{B})^2$ is convex, we get $g(t)\le tg(1)+(1-t)g(0)= tA+(1-t)B$ for all $t\in[0,1]$. Applying this observation with $A=4 \phi(v) / \diam(K)^2$ and $B=4 \phi(\bar v) / \diam(K)^2$ and using the concavity of $\phi$ completes the proof.
\end{proof}

We give in the next example an application of Proposition~\ref{prop:concave-profile} that gives a profile $\phi: K \rightarrow [0,+\infty)$ of a homogeneous distance on $\HH$ 
that is not continuous on $K$.

\begin{example} \label{ex:non-cont-profile} Let $\bb D :=\{v\in\R^2: \|v\| \leq 1\}$  denote the closed unit disc in $\R^2$ and let $\phi_1:\bb D\rightarrow [1/4,+\infty)$ be defined by
\begin{equation*} 
	\phi_1(x,y): = \frac14 + 1 - \frac{y^2}{1-x^2} 
\end{equation*}
for $(x,y) \in \bb D \setminus \{(-1,0),(1,0)\}$ and 
\begin{equation*} 
\phi_1(-1,0) := \frac{5}{4}, \quad \phi_1 (1,0) := \frac{5}{4}.
\end{equation*}
Then $\phi_1:\bb D\rightarrow [1/4,+\infty)$ is the profile of a homogeneous distance on $\HH$ but $\phi_1$ is not continuous on $\bb D$.

Indeed, the function $\phi_1$ is clearly continuous on $\bb D\setminus \{(-1,0),(1,0)\}$. Since $\phi_1(v) \leq 5/4$ for every $v\in \bb D$, $\phi_1$ is also clearly upper semicontinuous on $\bb D$. However, $\phi_1(x,y) = 1/4$ for every $(x,y) \in \bb S^1 \setminus \{(-1,0),(1,0)\}$, hence $\phi_1$ is not continuous at $(-1,0)$ and $(1,0)$. The fact that $\phi_1$ is the profile of a homogeneous distance will now follow from Proposition~\ref{prop:concave-profile} if we show that $\phi_1$ is concave on $\interior (\bb D)$. To this aim, we set $$f(x,y):=\frac{y^2}{1-x^2}$$ and we prove that $f$ is convex on $\interior (\bb D)$. We have
	\[
	\frac{\de^2 f}{\de x^2} = 2y^2 \frac{1+3x^2}{(1-x^2)^3} ,
	\qquad
	\frac{\de^2 f}{\de x\de y} = \frac{4xy}{(1-x^2)^2} ,
	\qquad
	\frac{\de^2 f}{\de y^2} = \frac{2}{1-x^2} .
	\]
	Hence the determinant and the trace of the Hessian $H_f$ of $f$ are given by $$\det(H_f) = \frac{4y^2}{(1-x^2)^3} \quad \text{and} \quad \operatorname{trace} (H_f) = 2y^2 \frac{1+3x^2}{(1-x^2)^3} +\frac{2}{1-x^2}$$  that are both non negative on $\interior (\bb D)$. Hence $f$ is convex on $\interior (\bb D)$ as claimed.
\end{example}

Note that it will follow from the necessary conditions proved in Section~\ref{sect:necessary-conditions} that the homogeneous distance with profile $\phi_1$ does not satisfy BCP, see Remark~\ref{rk:nobcp-non-cont-profile}. We refer to Example~\ref{ex:non-cont-profile-with-bcp} for another example of a homogeneous distance with non-continuous profile that satisfy BCP.

\medskip

We end this section with conventions and notations that will be used in the rest of the paper. We denote by the same notations, namely, $\langle \cdot ,\cdot \rangle$ and $\|\cdot\|$, the Euclidean scalar product and Euclidean norm on $\R^2$ or $\R^3$. For $n=2,3$, and $v, w\in\R^n \setminus\{0\}$, we denote by $\measuredangle (v,w)$ the unoriented angle between $v$ and $w$ given by
\begin{equation*}
\measuredangle (v,w) := \arccos \frac{\langle v , w \rangle}{\|v\| \|w\|} \in [0,\pi]~.
\end{equation*}

Given $v, w \in \R^2\setminus\{0\}$, the oriented angle $\angle (v,w) \in (-\pi, \pi)$ between $v$ and $w$ is well defined whenever $w\notin \R_{<0} \,v$ and is given by
\begin{align*}
\angle (v,w): =
\begin{cases}
\sgn (\omega(v,w)) \, \measuredangle (v,w) &\, \text{if } \omega(v,w) \not= 0~,\\
0 &\, \text{if } \omega(v,w) = 0 \text{ and }\langle v , w \rangle >0 ~.
\end{cases}
\end{align*}

We also recall that, given any two vectors $v, w \in \R^2\setminus\{0\}$ with $w\notin \R_{<0} \,v$, we have 
\begin{equation*} \label{e:omega-vs-sin}
\omega(v,w)  = \|v\| \|w\| \sin \angle (v,w).
\end{equation*}

We denote by $\bb S^1$ the Euclidean unit circle in $\R^2$.


\section{Sufficient conditions} 
\label{sect:sufficient-cdt}

This section is devoted to the proof and applications of Theorem~\ref{thm:intro-sufficient-cdt}, restated below for the reader convenience.  We refer to Proposition~\ref{prop:LDNG} for the definition and general properties of the profile of a homogeneous distance on $\HH$.

\begin{theorem} \label{thm:sufficient-conditions-bcp}
Let $d$ be a homogeneous distance on $\HH$ with profile $\phi:K\rightarrow [0,+ \infty)$. Assume that
\begin{equation}\label{e:phi-2} 
\begin{split}
\exists\,  m> & 0,\, \exists\,  \alpha \in (0,\pi / 2), \text{ for a.e.~} v\in \interior(K)\setminus\{0\}, \, \forall w\in \R^2\setminus\{0\},\\
& \measuredangle (v,w) \leq \alpha \Rightarrow \langle \nabla\phi(v) , w \rangle \leq -m \|v\| \|w\|, 
 \end{split}
 \end{equation}
and
\begin{equation} \label{e:phi-3}
\begin{split}
\exists\, \kappa> & 0,\, \exists\, M>0,\; \{v\in\R^2 : \|v\|\leq \kappa\} \subset \interior(K) \text{ and } \\  
&\|\nabla\phi(v)\| \leq M \|v\| \text{ for a.e.~} v\in \R^2 \text{ such that }\|v\|\leq \kappa.
\end{split}
 \end{equation}
Then BCP holds on $(\HH,d)$.
\end{theorem}

We first give some applications of Theorem~\ref{thm:sufficient-conditions-bcp} in Section~\ref{subsect:applications}. We prove in Section~\ref{subsect:abstract-sufficient-cdt} an abstract sufficient condition, Theorem~\ref{thm:abstract-sufficient-cdt}, for the validity of BCP on $\HH$, or more generally on homogeneous groups, see Remark~\ref{rmk:homogeneous-groups}, equipped with a homogeneous distance. The proof of Theorem~\ref{thm:sufficient-conditions-bcp}, or equivalently of Theorem~\ref{thm:intro-sufficient-cdt}, will then be given in Section~\ref{subsect:proof-main-sufficient} and will be based on a corollary of this abstract sufficient condition, Corollary~\ref{cor:abstract-sufficient-cdt}.

\subsection{Applications and Examples} \label{subsect:applications}

First, combining Proposition~\ref{prop:concave-profile} and Theorem~\ref{thm:sufficient-conditions-bcp} gives the following class of homogeneous distances that satisfy BCP on $\HH$. Namely, if $K\subset \R^2$ is a compact convex and symmetric set containing $0$ in its interior and $\phi: K \rightarrow [\diam(K)^2/16,+\infty)$  is concave on $\interior (K)$, upper semicontinuous on $K$, and satisfies~\eqref{e:phi-2} and~\eqref{e:phi-3}, then $\phi$ is the profile of a homogeneous distance that satisfy BCP on $\HH$. We give below an example of such a function that also illustrates the fact that such profiles need not be continuous on $K$. 

\begin{example} \label{ex:non-cont-profile-with-bcp} Let $\phi_2:\bb D\to [1/4,+\infty)$ be defined by 
$$\phi_2(x,y) := \phi_1(x,y) + 1 - x^2 - y^2$$
where $\phi_1$ is given by Example~\ref{ex:non-cont-profile}. It follows from the definition of $\phi_1$ that $\phi_2$ is continuous on $\bb D\setminus \{(-1,0),(1,0)\}$ but not at $(-1,0)$ and $(1,0)$. We prove now that $\phi_2$ is the profile of a homogeneous distance that satisfies BCP. First, $\phi_2$ is the sum of $\phi_1$ with a non negative, continuous, and concave function on $\bb D$. It follows from Example~\ref{ex:non-cont-profile} that $\phi_2$ is concave on $\interior(\bb D)$ and upper semicontinuous on $\bb D$, hence $\phi_2$ is the profile of a homogeneous distance by Proposition~\ref{prop:concave-profile}. Next, $\phi_2\in C^2(\interior(\bb D))$ with $\nabla \phi (0) = 0$, hence~\eqref{e:phi-3} holds. By Theorem~\ref{thm:sufficient-conditions-bcp}, to prove that BCP holds for the homogeneous distance whose profile is given by $\phi_2$, it is thus sufficient prove that~\eqref{e:phi-2} holds. 

Let $v\in \interior(\bb D) \setminus \{0\}$. Using polar coordinates, we write $v:=(r\cos \theta,r\sin\theta)$ for some $r\in (0,1)$ and $\theta \in [0,2\pi]$ and we let $w:=(\cos (\theta + \alpha), \sin (\theta + \alpha)) \in \bb S^1$ with $\alpha \in (-\pi/2,\pi/2)$. We have
$$\langle \nabla\phi_2(v) , w \rangle = \cos\alpha  \frac{\de\phi_2}{\de r}(r,\theta) + \sin\alpha \, \frac{1}{r} \frac{\de\phi_2}{\de\theta}(r,\theta). $$
Hence, to prove that~\eqref{e:phi-2} holds, it is sufficient to find  $\alpha_0 \in (0,\pi/2)$, such that
\begin{equation} \label{e:bcp-phi2}
\cos\alpha  \, \frac{1}{r} \frac{\de\phi_2}{\de r}(r,\theta) + \sin\alpha \, \frac{1}{r^2} \frac{\de\phi_2}{\de\theta}(r,\theta) \leq -1
\end{equation}
for every $r\in (0,1)$, $\theta \in [0,2\pi]$, and $\alpha \in [-\alpha_0,\alpha_0]$. We have
\begin{equation} \label{e:Phi_2-r}
\frac{\de\phi_2}{\de r}(r,\theta) = 
		-2r \left(1 + \frac{ \sin^2\theta }{ (1-r^2\cos^2\theta)^2 }\right),
\end{equation}
hence,
\begin{equation*} 
 \frac{\de\phi_2}{\de r}(r,\theta) \leq -2r
\end{equation*}
for every  $r\in (0,1)$ and $\theta \in [0,2\pi]$. Next, we have
\begin{equation} \label{e:Phi_2-theta}
\frac{\de\phi_2}{\de\theta}(r,\theta) = 
		- 2 r^2 \frac{(1-r^2) \cos\theta\sin\theta }{ (1-r^2\cos^2\theta)^2 }.
\end{equation}		
For $r\in (0,1/2]$ and $\theta \in [0,2\pi]$,
$$\left|\frac{1}{r^2} \, \frac{\de\phi_2}{\de\theta}(r,\theta) \right| \leq \frac{32}{9}.$$
Hence, one can find $\alpha_0 \in (0,\pi/2)$ such that
\begin{equation*}
\begin{split}
\cos\alpha  \, \frac{1}{r} \frac{\de\phi_2}{\de r}(r,\theta) + \sin\alpha \, \frac{1}{r^2} \frac{\de\phi_2}{\de\theta}(r,\theta) &\leq -2 \cos\alpha + \frac{32}{9} |\sin \alpha| \leq -1
\end{split}
\end{equation*}
for every $r\in (0,1/2]$, $\theta \in [0,2\pi]$, and $\alpha \in [-\alpha_0,\alpha_0]$. Next, it follows from~\eqref{e:Phi_2-r} and~\eqref{e:Phi_2-theta} that, for $r\in (0,1)$ and $\theta\in [0,2\pi]$, we have 
\begin{equation*}
\begin{split}
\left|\frac{\de\phi_2}{\de\theta}(r,\theta)\right| \cdot \left|\frac{\de\phi_2}{\de r}(r,\theta)\right|^{-1} &= \frac{r(1-r^2)|\cos\theta\sin\theta|}{(1-r^2\cos^2\theta)^2 + \sin^2\theta}\\
&\leq \frac{(1-r^2)|\sin\theta|}{(1-r^2)^2 + \sin^2\theta} \leq \frac{1}{2}.
\end{split}
\end{equation*}
Hence, one can find $\alpha_0 \in (0,\pi/2)$ such that, for every $r\in (1/2,1)$, $\theta \in [0,2\pi]$, and $\alpha \in [-\alpha_0,\alpha_0]$, 
\begin{equation*}
\begin{split}
\cos\alpha  \, \frac{1}{r} \frac{\de\phi_2}{\de r}(r,\theta) + \sin\alpha \, \frac{1}{r^2} &\frac{\de\phi_2}{\de\theta}(r,\theta) \\
&\leq \cos\alpha  \, \frac{1}{r} \frac{\de\phi_2}{\de r}(r,\theta) + 2 |\sin\alpha|  \, \frac{1}{r}\left|\frac{\de\phi_2}{\de r}(r,\theta)\right| \\
 &= \left(\cos\alpha - 2 |\sin\alpha| \right)  \, \frac{1}{r}\frac{\de\phi_2}{\de r}(r,\theta)\\
 &\leq -2 \left(\cos\alpha - 2 |\sin\alpha| \right) \leq -1.
\end{split}
\end{equation*}
This completes the proof of~\eqref{e:bcp-phi2} and hence of the fact that BCP holds for the homogeneous distance with profile $\phi_2$.
\end{example}

Homogeneous distances that satisfy the assumptions of Theorem~\ref{thm:sufficient-conditions-bcp} can also be constructed as follows. Let $K\subset \R^2$ be a compact convex set such that $0\in\interior(K)$ and $K=-K$. Let $\psi\in \Lip(K)$ and assume that $\psi$ satisfies~\eqref{e:phi-2} and~\eqref{e:phi-3}. By~\cite[Proposition 6.3]{LeDonne_Nicolussi_regularity_spheres}, there is a constant $C\geq 0$ such that, for every $c\geq C$, the function $\psi + c :K \rightarrow [0,+\infty)$ defines the profile of a homogeneous distance on $\HH$. Since $\psi +c$ satisfies~\eqref{e:phi-2} and~\eqref{e:phi-3} as well, it follows from Theorem~\ref{thm:sufficient-conditions-bcp} that such homogeneous distances satisfy BCP on $\HH$. This allows in particular to contruct another large class of homogeneous distances on $\HH$ for which BCP holds.

\medskip

Next, for rotationally invariant homogeneous distances, Theorem~\ref{thm:sufficient-conditions-bcp} can be rephrased as follows. We refer to the paragraph before Theorem~\ref{thm:intro-characterization-rot-inv-dist} for the definition of rotationally invariant homogeneous distances and of their radial profile.

\begin{corollary} \label{cor:sufficiention-rot-inv}
Let $d$ be a rotationally invariant homogeneous distance on $\HH$ with radial profile $\varphi:[-r_d,r_d] \rightarrow [0,+\infty)$.  Assume that $|\varphi'(s) / s|$ is essentially bounded in a neighbourhood of $0$ and that there is $m>0$ such that $\varphi'(s) \leq -m s$ for a.e.~$s\in (0,r_d)$. Then BCP holds in $(\HH,d)$.
\end{corollary} 

\begin{proof}
Let $\phi: \pi(B) \rightarrow [0,+\infty)$ and $\varphi:[-r_d,r_d] \rightarrow [0,+\infty)$ denote respectively the profile and  radial profile of $d$, so that $\phi(v) = \varphi(\|v\|)$. We have
$$\nabla\phi(v) = \varphi'(\|v\|)\, \frac{v}{\|v\|}$$
for a.e.~$v \in \interior(\pi(B))  \setminus \{0\}$. It follows that, for a.e.~$v \in \interior(\pi(B)) \setminus \{0\}$ and every $w\in \R^2\setminus \{0\}$ such that $\measuredangle (v,w) \leq \pi /4$,
\begin{equation*}
\begin{split}
\langle \nabla\phi(v) , w \rangle &= \frac{\varphi'(\|v\|)}{\|v\|} \, \langle v,w \rangle \\
&=\varphi'(\|v\|) \, \|w\| \cos \measuredangle (v,w) \\
& \leq - \frac{\sqrt 2}{2} m \|v\| \|w\|~,
\end{split}
\end{equation*}
and hence~\eqref{e:phi-2} holds. Next,~\eqref{e:phi-3} is a straightforward consequence of the fact that $\|\nabla\phi(v)\| = |\varphi'(\|v\|)|$ for a.e.~$v \in \interior(\pi(B)) \setminus \{0\}$ together with the assumption that  $|\varphi'(s) / s|$ is essentially bounded in a neighbourhood of $0$. Hence BCP holds in $(\HH,d)$ by Theorem~\ref{thm:sufficient-conditions-bcp}.
\end{proof}

Note that, since the radial profile $\varphi$ is even, then, if $\varphi \in C^1((-r_d,r_d))$ and $\varphi'$ is differentiable at $0$, we have $\varphi'(0)=0$ and $|\varphi'(s) / s|$ is bounded in a neighbourhood of $0$. This leads to the following corollary.

\begin{corollary} \label{cor:sufficiention-rot-inv-bis} Let $d$ be a rotationally invariant homogeneous distance on $\HH$ with radial profile $\varphi:[-r_d,r_d] \rightarrow [0,+\infty)$. Assume $\varphi \in C^1((-r_d,r_d))$, $\varphi'$ is differentiable at $0$, and there is $m>0$ such that $\varphi'(s) \leq -m s$ for every $s\in (0,r_d)$. Then BCP holds in $(\HH,d)$.
\end{corollary} 

Corollary~\ref{cor:sufficiention-rot-inv-bis} gives the "if" part in Theorem~\ref{thm:intro-characterization-rot-inv-dist}. The validity of BCP for the rotationally invariant homogeneous distances $d_\epsilon$~\eqref{e:d_epsilon} when $\epsilon >0$ comes as an application of Corollary~\ref{cor:sufficiention-rot-inv-bis}. Indeed, let $\epsilon >0$ be fixed and let $B_\epsilon$ denote the unit ball centered at the origin in $(\HH,d_\epsilon)$. One can easily check that $\pi(B_{\epsilon})$ is a Euclidean disc with radius $r_\epsilon:=1/\sqrt{1+\epsilon^2}$ and that the radial profile $\varphi_\epsilon$ of $d_\epsilon$ is given by 
\[
\varphi_\epsilon(s) := \frac{1}{4} \left(\left(1-\epsilon^2 s^2\right)^2-s^4\right)^{1/2}
\]
for $s\in [-r_\epsilon,r_\epsilon]$. We have 
\[
\varphi_\epsilon'''(s)= \frac{3 s \left(\left(\epsilon^4-1\right) s^4-1\right)}{\left(\left(1-\epsilon^2 s^2\right)^2-s^4\right)^{5/2}}  < 0,
\]
and hence, $\varphi_\epsilon''(s) \leq \varphi_\epsilon''(0) = -\epsilon^2 / 2$ for every $s\in (0,r_\epsilon)$ . Therefore one can apply Corollary~\ref{cor:sufficiention-rot-inv-bis} with $m =\epsilon^2 / 2$ and BCP holds in $(\HH,d_\epsilon)$ for every $\epsilon >0$ as claimed.

\begin{remark} The proof of Theorem~\ref{thm:sufficient-conditions-bcp} and Corollary~\ref{cor:sufficiention-rot-inv-bis} can also be used to prove the validity of BCP for the rotationally invariant continuous homogeneous quasi-distance on $\HH$ defined by $$d(0,(v,z)): = (\|v\|^2 + |z|)^{1/2}$$ whose unit ball centered at the origin can be described as 
$$\{(v,z)\in \HH:\ \|v\|\leq 1,\ -\varphi(\|v\|)  \le z\le \varphi(\|v\|) \}$$ with $\varphi(s) := 1 - s^2$ for $s\in [-1,1]$.
\end{remark}

\begin{remark} Corollary~\ref{cor:sufficiention-rot-inv-bis} can be used to recover the fact that the rotationally invariant homogeneous distances $d_\alpha$ considered in~\cite{LeDonne_Rigot_Heisenberg_BCP} satisfy BCP. Their radial profile $\varphi_\alpha: [-\alpha,\alpha] \rightarrow[0,+\infty)$ is indeed given by 
\begin{equation*}
\varphi_\alpha (s):= \left(\alpha^2 - s^2\right)^{1/2} ,
\end{equation*}
which satisfies $\varphi'_\alpha (s) = - s \left(\alpha^2 - s^2\right)^{-1/2} \leq - \alpha^{-1} s$ for every $s\in (0,\alpha)$.
\end{remark}

\subsection{An abstract sufficient condition}  \label{subsect:abstract-sufficient-cdt}
We recall that a subset $A$ of metric space $(X,\rho)$ is said to be an \textit{$\epsilon$-net} if $\rho(x,y)\ge\epsilon$ for all $x, y \in A$, $x\not= y$, and that $(X,\rho)$ is said to be \emph{$\epsilon$-compact} if there is a constant $C\geq 1$ such that $\card A \leq C$ for every $\epsilon$-net $A$ in $(X,\rho)$. 

\begin{theorem} \label{thm:abstract-sufficient-cdt}
	Let $d$ be a homogeneous distance on $\HH$. Assume that there are $\epsilon \in (0,1)$, $\epsilon'>0$, an $\epsilon'$-compact metric space $(X,\rho)$ and a map $\sigma:\HH\setminus\{0\}\to X$ such that $\sigma(\delta_t (p))=\sigma(p)$ for every $t>0$, $p\in\HH\setminus\{0\}$, and such that
\begin{equation} \label{e:sufficent-for-bcp}
	\forall p\in \partial B,\ \forall q\in \delta_{\epsilon} (B)\setminus B(p,1),
	\ \rho(\sigma(p),\sigma(q)) \ge \epsilon'~.
	\end{equation}
	Then $(\HH,d)$ satisfies BCP. 
\end{theorem}

\begin{proof} Let $\mathcal{B} := \{B(p_i,r_i)\}_{i=1}^n$ be a finite Besicovitch family of balls in $(\HH,d)$, see Definition~\ref{def:Besicovitchfamily}, and assume that $r_1\geq \dots\geq r_n >0$. By left-invariance, we may also assume that $0\in\cap_{i=1}^n  B(p_i,r_i)$, and, with no loss of generality, that $r_i = d(0,p_i)$ for every $1\leq i \leq n$. 

We set $i_1:=1$ and then iteratively $i_{k+1}:=\min\{j>i_k : r_j < \epsilon r_{i_k}\}$ as long as possible. We get from this construction a finite sequence $i_1=1<\dots<i_l\leq n$ for some integer $l\geq 1$ with the following properties. Set $i_{l+1}:=n+1$ and $I_k:=\{i_k,\dots,i_{k+1}-1\}$ for $k\in\{1,\dots,l\}$. First, for all $k\in\{1,\dots,l\}$ and all $j\in I_k$, $\epsilon r_{i_k} \leq r_j \leq r_{i_k}$. Second, $r_{i_{k+1}} < \epsilon r_{i_k}$ for all $k\in\{1,\dots,l-1\}$ whenever $l\geq 2$. Third, $\cup_{k=1}^l I_k = \{1,\dots,n\}$.


Next, we prove that, for all $k\in\{1,\dots,l\}$, $\card I_k \leq C$ where $C$ is the maximal cardinality of an $\epsilon$-net in $B$. Recall that $B$ is a compact subset of $(\HH,d)$ and hence is $\epsilon$-compact. For all $j\in I_k$, $d(0,p_j)= r_j\leq r_{i_k}$ and, for all $i$, $j \in I_k$, $i\not= j$, $d(p_i,p_j) > \max(r_i,r_j) \geq \epsilon r_{i_k}$. Hence $\{\delta_{1/r_{i_k}}(p_j)\}_{j\in I_k}$ is an $\epsilon$-net in $B$ and the claim follows.

Finally, we prove that $l\leq C'$ where $C'$ is the maximal cardinality of an $\epsilon'$-net in $(X,\rho)$. To this aim, it is sufficient to check that $\{\sigma(p_{i_k})\}_{k=1}^l$ is an $\epsilon'$-net in $(X,\rho)$ whenever $l\geq 2$. Let $k,h\in\{1,\dots,l\}$ with $h>k$. We have 
$$d(\delta_{1/r_{i_k}}(p_{i_h}), 0) = r_{i_k}^{-1}\, r_{i_h} < \epsilon^{h-k}\leq \epsilon$$ 
and 
$$d(\delta_{1/r_{i_k}} (p_{i_h}),\delta_{1/r_{i_k}} (p_{i_k})) = r_{i_k}^{-1}\, d(p_{i_h},p_{i_k}) >r_{i_k}^{-1}  \max (r_{i_k},r_{i_h}) = 1~.$$ 
Hence $\delta_{1/r_{i_k}} (p_{i_h}) \in \delta_\epsilon (B)\setminus B(\delta_{1/r_{i_k}} (p_{i_k}),1)$ with $\delta_{1/r_{i_k}} (p_{i_k}) \in \partial B$. Since $\sigma \circ \delta_t =  \sigma$ for every $t>0$, it follows from~\eqref{e:sufficent-for-bcp} that
$$\rho(\sigma(p_{i_k}),\sigma(p_{i_h}))  = \rho(\sigma(\delta_{1/r_{i_k}} (p_{i_k})),\sigma(\delta_{1/r_{i_k}} (p_{i_h}))) \geq \epsilon'~$$ which proves the claim.

All together, we get $\card \mathcal{B} = \sum_{k=1}^j \card I_k \leq CC'$, and hence BCP holds in $(\HH,d)$.
\end{proof}

\begin{remark} \label{rmk:homogeneous-groups} The proof of Theorem~\ref{thm:abstract-sufficient-cdt} only uses properties that the Heisenberg group shares with any other homogeneous group equipped with a homogeneous distance. Hence the abstract sufficient condition for the validity of BCP given in Theorem~\ref{thm:abstract-sufficient-cdt} holds more generally for homogeneous distances on homogeneous groups.
\end{remark}

The proof of Theorem~\ref{thm:sufficient-conditions-bcp} will be based on the following corollary of Theorem~\ref{thm:abstract-sufficient-cdt}.

\begin{corollary} \label{cor:abstract-sufficient-cdt}
Let $d$ be a homogeneous distance on $\HH$. Assume that there is $\epsilon\in (0,1)$ such that, if $p=(v,z)\in \partial B$, $q=(v',z')\in \partial B$ are such that $v\not= 0$, $v'\not= 0$, $\|v'/\|v'\| -v/\|v\| \|<\epsilon$ and $|z'-z|<\epsilon$, then $\delta_t(q)\in B(p,1)$ for all $t\in[0,\epsilon]$. Then BCP holds in $(\HH,d)$. 
\end{corollary}

\begin{proof} To apply Theorem~\ref{thm:abstract-sufficient-cdt}, we introduce a compact metric space $(X, \rho)$ as follows. Let $M:=\max\{z>0:\, (v,z)\in B \text{ for some } v\in \R^2\}$ and $z_N>0$ be such that $d(0,(0,z_N))=1$. We set $X: = \bb S^1 \times [-M,M] \cup \{(0,z_N),(0,-z_N)\}$ and we equip $X$ with the distance $\rho((v,z),(v',z')): = \max (\|v'-v\|, |z'-z|)$. Then $(X,\rho)$ is compact and hence $\epsilon'$-compact for all $\epsilon' >0$.

Define $\hat\pi:\partial B \setminus \{0\} \rightarrow X$ by $\hat\pi(0,z_N):= (0,z_N)$, $\hat\pi(0,-z_N):= (0,-z_N)$, and $\hat\pi(v,z) := (v/\|v\|,z)$ if $v\not= 0$. Let $\sigma:\HH\setminus \{0\} \rightarrow X$ be defined by $\sigma(p) := \hat\pi (\delta_{\R_{>0}} (p) \cap \partial B)$. We have $\sigma(\delta_t (p))=\sigma(p)$ for all $t>0$, $p\in\HH\setminus \{0\}$. 

To conclude the proof, let us check that~\eqref{e:sufficent-for-bcp} holds with $\epsilon \in (0,1)$ given by the assumption in Corollary~\ref{cor:abstract-sufficient-cdt} and $\epsilon':= \min(2 z_N,\epsilon)$. We set $p_N:=(0,z_N)\in \HH$. Let $p\in \partial B$ and $q\in \delta_\epsilon (B)\setminus B(p,1) $. First, if $p=p_N$, then $\sigma(q) \in X \setminus \{(0,z_N)\}$. Indeed, otherwise $q=(0,z)$ for some $z>0$. On the one hand, this implies that $d(p,q)=\sqrt{|1-z/z_N|} >1$ since $q\not\in B(p,1)$. On the other hand, since $q \in \delta_\epsilon (B)$, we have $d(0,q) = \sqrt{|z/z_N|} \leq \epsilon $, and hence, $0<z/z_N<1$. This implies that $d(p,q)=\sqrt{1-z/z_N} < 1$ which gives a contradiction. Similarly, if $p=-p_N$, then $\sigma(q) \in X \setminus \{(0,-z_N)\}$. Hence $\rho(\sigma(p), \sigma(q)) \geq \min(1,2 z_N)\geq \epsilon'$ when $p=p_N$ or $p=-p_N$. Next, assume that $p=(v,z)$ with $v\not= 0$ and assume by contradiction that $\rho(\sigma(p),\sigma(q)) < \epsilon'$. This implies that $q' := \delta_{1/d(0,q)}(q)\in \partial B$ can be written as $q' =(v',z')$ for some $v'\not= 0$ such that $\|v'/\|v'\| -v/\|v\| \|<\epsilon'\leq \epsilon$ and some $z'$ such that $|z'-z|<\epsilon'\leq\epsilon$. Moreover, we have $t:=d(0,q)\leq \epsilon$ and hence, we get by assumption that $ \delta_t(q')=q\in B(p,1)$. This gives a contradiction and concludes the proof.
\end{proof}

%
\subsection{Proof of Theorem~\ref{thm:sufficient-conditions-bcp}}  \label{subsect:proof-main-sufficient} From now on in this section we assume that $d$ is a homogeneous distance on $\HH$ that satisfies the assumptions of Theorem~\ref{thm:sufficient-conditions-bcp}. The validity of BCP will follow as an application of Corollary~\ref{cor:abstract-sufficient-cdt}. Namely, we will prove the following lemma.

\begin{lemma} \label{lem:sufficient-for-bcp}
There are $\epsilon \in (0,1)$ and $\theta_\epsilon \in (0,\pi/2)$ such that if $p_0=(v_0,z_0) \in \partial B$ with $v_0 \not= 0$ and $z_0\geq 0$ and $q = (v,z) \in \partial B$ with $v\not= 0$ are such that $|z-z_0| < \epsilon$ and $\measuredangle (v_0,v) <  \theta_\epsilon$, then $p_0^{-1}\cdot \delta_t(q) \in B$ and $\delta_t(q)\cdot p_0^{-1} \in B$ for every $t\in [0,\epsilon]$.
\end{lemma}

Before giving the proof of Lemma~\ref{lem:sufficient-for-bcp}, we first explain how Theorem~\ref{thm:sufficient-conditions-bcp} follows from Lemma~\ref{lem:sufficient-for-bcp} in conjonction with Corollary~\ref{cor:abstract-sufficient-cdt}.

\medskip

\noindent\textit{Proof of Theorem~\ref{thm:sufficient-conditions-bcp}.} Let $\epsilon\in (0,1)$, $\theta_\epsilon \in (0,\pi/2)$ be given by Lemma~\ref{lem:sufficient-for-bcp}, and set $\epsilon' := \min (\epsilon,\sqrt{2}\,(1-\cos \theta_\epsilon))$. Let $p_0=(v_0,z_0) \in \partial B$, $q = (v,z) \in \partial B$ be such that $v_0 \not= 0$, $v\not= 0$, $|z-z_0| < \epsilon'$, and $\|v/\|v\| -v_0/\|v_0\| \|<\epsilon'$. Then $\measuredangle (v_0,v) < \theta_\epsilon$. If $z_0 \geq 0$, the fact that $\delta_t(q) \in B(p_0,1)$ for every $t\in [0,\epsilon']$ follows from the first conclusion in Lemma~\ref{lem:sufficient-for-bcp}. If $z_0 <0$, then $p_0^{-1} = (-v_0,-z_0)\in \partial B$ with $-z_0>0$, $q^{-1}=(-v,-z)$, $-v_0 \not= 0$, $-v\not= 0$, $|-z+z_0| < \epsilon'$, and $\measuredangle (-v_0,-v) =  \measuredangle (v_0,v)  < \theta_\epsilon$. Therefore it follows from the second conclusion in Lemma~\ref{lem:sufficient-for-bcp} that $\delta_t(q^{-1})\cdot p_0 \in B$, and hence $(\delta_t(q^{-1})\cdot p_0)^{-1} =  p_0^{-1}\cdot \delta_t(q) \in B$, for every $t\in [0,\epsilon']$. Hence Corollary~\ref{cor:abstract-sufficient-cdt} can be applied and BCP holds in $(\HH,d)$ as claimed. \hfill $\Box$

\medskip

The rest of the section is devoted to the proof of Lemma~\ref{lem:sufficient-for-bcp}. It will be achieved through a series of lemmas, from Lemma~\ref{lem:prop-of-phi} to Lemma~\ref{lem:g-2} below.

\medskip

We say that a constant is an absolute constant if its value can be chosen depending only on the parameters $m$, $\alpha$, $\kappa$ and $M$ given by~\eqref{e:phi-2} and~\eqref{e:phi-3}, geometric properties of the set $K$ and properties of the function $\phi$. We refer to Proposition~\ref{prop:LDNG} for these latter properties. In particular $K\subset \R^2$ is bounded and we let $D>0$ be such that 
\begin{equation} \label{e:def-D}
K \subset \{v\in \R^2 : \|v\| \leq D \}.
\end{equation}
We also set 
\begin{equation} \label{e:def-eta}
\eta := \min \{\phi(v) : \|v\| \leq \kappa \}>0
\end{equation}
where $\kappa$ is given by~\eqref{e:phi-3}. 

\medskip

We start with three preliminary lemmas, Lemma~\ref{lem:prop-of-phi}, Lemma~\ref{lem:a.e.vs-every} and Lemma~\ref{lem:v_0-v}.

\begin{lemma}\label{lem:prop-of-phi}
	There is an absolute constant $\epsilon_0 \in (0,1)$ such that for every $\epsilon \in (0,\epsilon_0)$, the following holds. Let $v,w\in K\setminus\{0\}$. Assume that $\measuredangle(v,w)\le\epsilon$ and $\|v\|-\|w\| \ge \sqrt{2\epsilon / m}$. Then $\phi(w)-\phi(v)\ge\epsilon$.
\end{lemma}

\begin{proof}
By~\eqref{e:radial-continuity-phi} and the continuity of $\phi$ on $\interior (K)$, it is sufficient to find an absolute constant $\epsilon_0 \in (0,1)$ such that for every $\epsilon \in (0,\epsilon_0)$, the following holds. For every $v \in \interior (K)\setminus\{0\}$ and a.e.~$w\in \interior (K) \setminus\{0\}$, 
\begin{equation*}
\measuredangle (v,w) \leq \epsilon \;  \text{ and } \; \|v\|-\|w\| \ge \sqrt{\frac{2\epsilon}{m}} \,\, \Rightarrow \,\, \phi(w)-\phi(v)\ge\epsilon~.
\end{equation*}

For $v,w\in \interior (K)$, we set $\sigma(t):=v+t(w-v)$. Since $\phi \in \Lip_{loc}(\interior(K))$ and~\eqref{e:phi-2} holds, we have that for every $v\in \interior (K)\setminus\{0\}$, the following properties hold for a.e.~$w\in \interior (K)\setminus\{0\}$. First, The map $t\mapsto \phi(\sigma(t))$ is differentiable a.e.~on $[0,1]$ with derivative given by $\langle \nabla\phi(\sigma(t)),\sigma'(t) \rangle$. Second, for a.e.~$t\in [0,1]$, we have, for every $u \in \R^2\setminus \{0\}$,
\begin{equation} \label{e:phi-2-a.e.}
\measuredangle (\sigma(t),u) \leq \alpha \,\, \Rightarrow \,\,\langle \nabla\phi(\sigma(t)) , u \rangle \leq -m \|\sigma(t)\| \|u\|~.
\end{equation}

We thus consider $v\in \interior (K)\setminus\{0\}$ and let $w\in \interior (K)\setminus\{0\}$ be such that these properties hold, and we assume that $\measuredangle(v,w)\le\epsilon$ and $\|v\|-\|w\| \ge \sqrt{2\epsilon / m}$ for some $\epsilon >0$ to be chosen small enough later.

\smallskip
	
First, if $\epsilon \leq \alpha /2$ is chosen small enough, we have
\begin{equation} \label{e:upper-bound-w/v}
\|w\| \leq \|v\|\, \frac{\sin(\alpha/2)}{\sin(\alpha/2+\epsilon)}~.
\end{equation}
Indeed otherwise $\|w\| >  \|v\|\sin(\alpha/2) / \sin(\alpha/2+\epsilon)$ and 
	\[
 \|v\|-\|w\| < \|v\| \left( 1 - \frac{\sin(\alpha/2)}{\sin(\alpha/2+\epsilon)} \right) .
	\]
	We have 
	\[
	\frac{\sin(\alpha/2)}{\sin(\alpha/2 + \epsilon)} = 1 - \frac{\cos(\alpha/2)}{\sin(\alpha/2)} \epsilon + o(\epsilon).
	\]
	Hence, if $\epsilon>0$ is chosen small enough, we get 
	\[
	\|v\|-\|w\| < 2 D  \frac{\cos(\alpha/2)}{\sin(\alpha/2)} \epsilon<  \sqrt{\frac{2\epsilon}{m}} ,
	\]
where $D$ is given by~\eqref{e:def-D}, which gives a contradiction.

\begin{figure}[ht]
	\includegraphics[scale = 0.8]{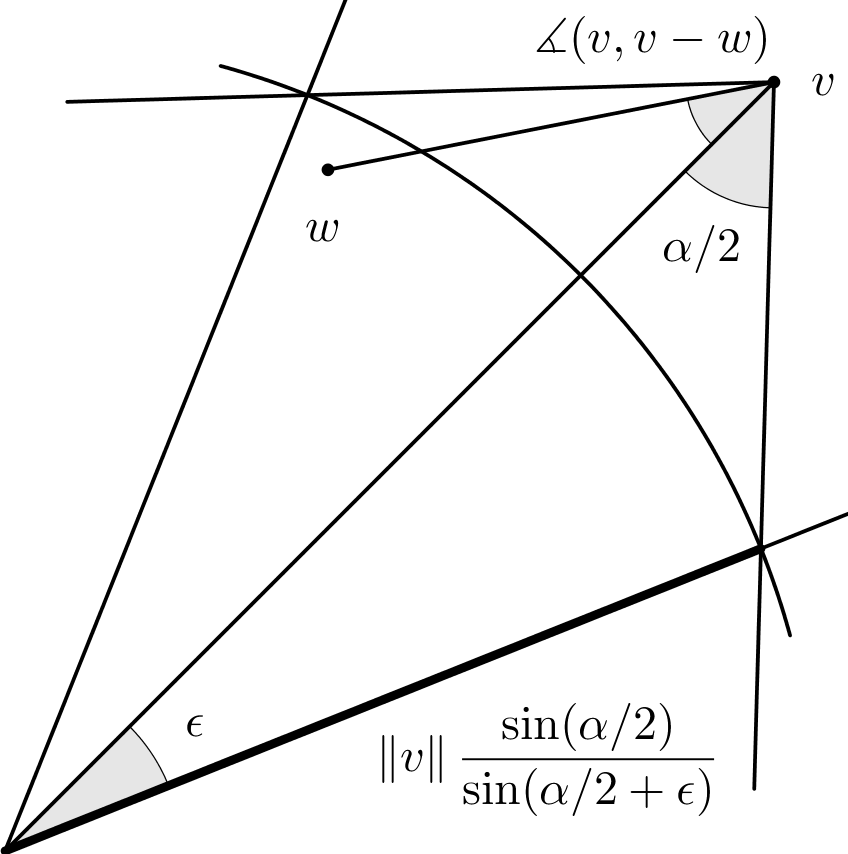}
	\caption{$\measuredangle (v,v-w) \leq \alpha/2$ } \label{fig:formulasin}
	\end{figure}

Next, it follows from elementary computations from planar geometry that the assumption $\measuredangle(v,w)\le\epsilon$ together with~\eqref{e:upper-bound-w/v} imply $\measuredangle (v,v-w) \leq \alpha/2$ and $\|w-v\| \leq \|v\|$ provided $\epsilon >0$ is chosen small enough, see Figure~\ref{fig:formulasin}. Recalling that $\sigma(t)=v+t(w-v)$, we get
\begin{equation*}
\begin{split}
	\measuredangle(\sigma(t),-\sigma'(t))  &= \measuredangle(\sigma(t),v-w) \\
	&= \measuredangle(\sigma(t),v) + \measuredangle(v,v-w)\\
	& \leq \measuredangle(w,v) + \alpha / 2 \leq \alpha
\end{split}
\end{equation*}
for every $t\in [0,1]$. Together with~\eqref{e:phi-2-a.e.}, this implies
	\begin{align*}
	\phi(w)-\phi(v) &= \int_0^1 \langle \grad\phi(\sigma(t)),\sigma'(t) \rangle \dd t \\
	&\ge \int_0^1 m \|\sigma(t)\| \|\sigma'(t)\| \dd t \\
	&\ge \int_0^1 m (\|v\|-t\|w-v\|) \|w-v\| \dd t \\
	&= m \|w-v\| \,(\|v\| - \frac12 \|w-v\|).
\end{align*}
Since $\|w-v\| \leq \|v\|$, we get 
\begin{equation*}
\phi(w)-\phi(v)  \ge \frac{m}{2} \|w-v\|^2
\ge \frac{m}{2} (\|v\|-\|w\|)^2 \ge \epsilon
	\end{equation*}
which concludes the proof.
\end{proof}

For $\epsilon \in (0,1)$, $\theta \in (-\pi,\pi)$ and  $p_0=(v_0,z_0) \in \partial B$ with $v_0\not= 0$, we set
\begin{equation*}
\partial (\epsilon,\theta,p_0) := \{(v,z) \in \partial B :\, v\not= 0,\, \angle (v_0,v) = \theta,\, |z-z_0| < \epsilon\}.
\end{equation*}
For $\nu \in \mathbb{S}^1$, we define 
\begin{equation} \label{e:def-hat-r}
\hat r(\nu) := \max\{s>0:\, s \nu \in K\}.
\end{equation}
Since $K$ is a compact convex set and $0\in \interior (K)$, the map $\nu \in \mathbb{S}^1 \mapsto \hat r(\nu)\in (0,+\infty)$ is well-defined and continuous. Hence $\hat r$ is uniformly continuous on $\mathbb{S}^1$ and it follows that, for every $\epsilon \in (0,1)$, there is $\hat\theta_\epsilon>0$ such that 
\begin{equation} \label{e:unif-cont-rtheta}
\measuredangle (\nu,\nu') \leq \hat\theta_\epsilon \, \Rightarrow\, |\hat r(\nu) - \hat r(\nu')| \leq \epsilon.
\end{equation}

The next lemma would be quite obvious if the profile $\phi$ were known to be continuous on $K$. However, as already stressed, in full generality $\phi$ is only known to be continuous on $\interior (K)$ and to satisfy~\eqref{e:radial-continuity-phi}, see for instance Example~\ref{ex:non-cont-profile} and Example~\ref{ex:non-cont-profile-with-bcp}.

\begin{lemma} \label{lem:a.e.vs-every}
For every $\epsilon \in (0,1)$, $\theta \in (-\pi/2, \pi/2)$, $p_0=(v_0,z_0) \in \partial B$ with $v_0\not=0$, and $q=(v,z) \in \partial (\epsilon,\theta,p_0)$, there is a positive sequence $(\tau_k)$ converging to $0$ such that, for every $k\geq 1$ and $\tau \in (0,\tau_k)$, there is $q'=(v',z') \in  \partial (\epsilon,\theta + \tau, p_0)$ such that $\|v' - v \| < 1/k$ and $|z'-z| < 1/k$. 
\end{lemma}

\begin{proof} Assume with no loss of generality that $q=(v,z) \in \partial (\epsilon,\theta,p_0)$ with $z\geq 0$. Otherwise one may argue considering $q^{-1} \in \partial (\epsilon,\theta,p_0^{-1})$. We divide the proof into two cases.

\smallskip

\noindent Case 1: If $z=\phi(v)$. By~\eqref{e:radial-continuity-phi} and the continuity of $\phi$ on $\interior(K)$, we have $\phi(sv) \rightarrow \phi(v)$ when $s\rightarrow 1$ and one can find a positive increasing sequence $(s_k)$ converging to 1 such that, for every $k\geq 1$, 
\begin{equation*}
0<D(1-s_k)  < \frac{1}{2k} ,
\end{equation*}
where $D$ is given by~\eqref{e:def-D},
\begin{equation*}
|\phi(s_k v) - \phi(v) | < \frac{1}{k},
\end{equation*}
and
\begin{equation*}
|\phi(s_k v) - z_0| < \epsilon.
\end{equation*}
By continuity of $\phi$ at $s_k v\in \interior(K)$, one can find $\omega_k \in (0,1/2k)$ such that
\begin{equation*}
\|v' - s_k v \| < \omega_k \Rightarrow v'\in \interior(K), \, |\phi(v') - \phi(v) | < \frac{1}{k} \text{ and } |\phi(v') - z_0| < \epsilon.
\end{equation*}
On the other hand, one can find $\tau_k >0$ such that, for every $\tau \in (0,\tau_k)$,
\begin{equation} \label{e:delta_k}
\|v'\| = s_k\| v\| \text{ and } \angle (v,v') = \tau \Rightarrow \|v' - s_k v \| < \omega_k.
\end{equation}
Then, choosing $v'$ such that $\|v'\|=s_k\|v\|$ and $\angle (v,v') = \tau$, and setting $z':=\phi(v')$, we get that $q':=(v',z') \in \partial (\epsilon, \theta+\tau, p_0)$ with 
\begin{equation*}
\|v'-v\| \leq \|v'-s_k v \| + \|s_k v - v\| < \omega_k + D(1-s_k)  < \frac{1}{k},
\end{equation*}
and $|z'-z| = |\phi(v') - \phi(v)| < 1/k$. 

\smallskip

\noindent Case 2: If $0\leq z < \phi(v)$, then $v\in \partial K$. By~\eqref{e:radial-continuity-phi} one can find a positive increasing sequence $(s_k)$ converging to 1 such that, for every $k\geq 1$,
\begin{equation*}
0<2D (1 - s_k) <\frac{1}{3k} ,
\end{equation*}
where $D$ is given by~\eqref{e:def-D}, and
\begin{equation*}
0\leq z < \phi(s_k v).
\end{equation*}
By continuity of $\phi$ at $s_k v\in \interior(K)$, one can find $\omega_k \in (0,1/3k)$ such that
\begin{equation} \label{e:omega_k}
\|v' - s_k v \| < \omega_k \Rightarrow v'\in \interior(K) \text{ and } 0\leq z < \phi(v')
\end{equation}
and one can find $\tau_k >0$ such that, for every $\tau \in (0,\tau_k)$,~\eqref{e:delta_k} holds. Then, for $\tau \in (0,\tau_k)$, we let $u_\tau\in \mathbb{S}^1$ be such that $\angle (v , u_\tau ) = \tau$ and we set $v_\tau := \hat r (u_\tau) u_\tau \in \partial K$.

\smallskip

If $0\leq z \leq \phi(v_\tau)$, then $(v_\tau,z) \in \partial(\epsilon,\theta + \tau,p_0)$ and we set $q':=(v',z')$ with $v':=v_\tau$ and $z':=z$. Since $v\in \partial K$, we have $\hat r(v/\|v\|) = \|v\|$ and $v = \hat r(v/\|v\|) v/\|v\|$. Hence 
\begin{equation*}
\begin{split}
\|v' - v\| &\leq \| \hat r (u_\tau) u_\tau - \hat r(v/\|v\|)u_\tau \| + \hat r (v/\|v\|) \| u_\tau - \frac{v}{\|v\|}\| \\
&\leq |\hat r (u_\tau) - \hat r(v/\|v\|)| + \tau \|v\|.
\end{split}
\end{equation*} 
Choosing $\tau_k$ small enough, we get by~\eqref{e:unif-cont-rtheta} that, for $\tau \in (0,\tau_k)$, $\|v' - v\| < 1/k$.

If $\phi(v_\tau) < z$, then, since $\phi(s_k \|v\| u_\tau) >z$ by~\eqref{e:delta_k} and~\eqref{e:omega_k}, and since $\phi$ is continuous on the segment from $s_k \|v\| u_\tau$ to $v_\tau$, one can find $t\in (s_k \|v\|,\hat r (u_\tau))$ such that $\phi(t u_\tau) = z$. Then we set $q':=(v',z')$ with $v':= t u_\tau$ and $z':=\phi(t u_\tau)$. We have $q'\in \partial(\epsilon, \theta + \tau, p_0)$ with 
\begin{equation*}
\begin{split}
\|v' - v\| &\leq \| t u_\tau- s_k \|v\| u_\tau \| + \|s_k \|v\| u_\tau  - s_k v\| +\|s_k v - v\|\\
&\leq (t - s_k \|v\|) + \omega_k + (1-s_k) \|v\|\\
&\leq \hat r (u_\tau) - s_k \hat r (v/\|v\|) + \omega_k + (1-s_k) \|v\|\\
&\leq \hat r (u_\tau) - \hat r (v/\|v\|) + (1-s_k) \hat r (v/\|v\|) + \omega_k + (1- s_k) \|v\|\\
&\leq \hat r (u_\tau) - \hat r (v/\|v\|) + 2 D (1-s_k)  + \omega_k \\
&\leq \hat r (u_\tau) - \hat r (v/\|v\|) + \frac{2}{3k}~.
\end{split}
\end{equation*}
Hence, choosing $\tau_k$ small enough, we get by~\eqref{e:unif-cont-rtheta} that, for $\tau \in (0,\tau_k)$, $\|v' - v \| < 1/k$.
\end{proof}

From now on, we fix $p_0=(v_0,z_0) \in \partial B$ with $v_0\not= 0$ and $z_0\geq 0$.  By a change of orthonormal coordinates that preserves the orientation in $\R^2$, we can assume with no loss of generality that $v_0=(a_0,0)$ for some $a_0>0$. Then, for $\theta \in \R$, we set $\nu_\theta := (\cos \theta, \sin \theta)\in \mathbb{S}^1$, so that $v_0 = a_0 \nu_0$ and $\angle (v_0,\nu_\theta) = \theta$, and we set $\hat r (\theta) := \hat r (\nu_\theta)$ where $\hat r$ is defined in~\eqref{e:def-hat-r}. Given $\epsilon  \in (0,1)$, we set $\theta_\epsilon := \min(\epsilon,\hat\theta_\epsilon)$ where $\hat \theta_\epsilon$ is given by~\eqref{e:unif-cont-rtheta}.

\medskip

For $x, y\in \R$, we use the notation $x\lesssim y$ to mean that there is an absolute constant $C>0$ such that $x\leq C y$.

\medskip

The next lemma gives a bound on $|\|v_0\| - \|v\||$ in terms of $\sqrt{\epsilon}$ for $q=(v,z) \in \partial (\epsilon,\theta,p_0)$, provided $\epsilon \in (0,\epsilon_0)$ and $\theta \in (-\theta_\epsilon,\theta_\epsilon)$, where $\epsilon_0$ is given by Lemma~\ref{lem:prop-of-phi}.

\begin{lemma} \label{lem:v_0-v}
For every $\epsilon \in (0,\epsilon_0)$, $\theta \in (-\theta_\epsilon,\theta_\epsilon)$, and $q=(v,z) \in \partial (\epsilon,\theta,p_0)$, we have $|\|v_0\| - \|v\|| \lesssim \sqrt{\epsilon}$.
\end{lemma}

\begin{proof} Let $\epsilon \in (0,\epsilon_0)$, $\theta \in (-\theta_\epsilon,\theta_\epsilon)$, and  $q=(v,z) \in \partial (\epsilon,\theta,p_0)$. We write $v=r\nu_\theta$ for some $r>0$. To get the conclusion of the lemma, we need to prove that 
\begin{equation} \label{e:a-r}
	|a_0-r| \lesssim \sqrt{\epsilon}~.
	\end{equation} 
We divide the proof of~\eqref{e:a-r} into five cases.

	\smallskip
	
	\noindent Case 1: If $z_0 = \phi(v_0)$ and $z=\phi(r\nu_\theta)$ then~\eqref{e:a-r} follows from Lemma~\ref{lem:prop-of-phi}.
	
	\smallskip
	
	\noindent Case 2: If $0\leq z_0 < \phi(v_0)$ and $z=\phi(r\nu_\theta)$, we have $\epsilon > z-z_0 > \phi(r\nu_\theta) - \phi(v_0)$. Then it follows from Lemma~\ref{lem:prop-of-phi} that $a_0-r \lesssim \sqrt{\epsilon}$. If $r\leq a_0$, this implies~\eqref{e:a-r}. If $r>a_0$, we have $a_0=\hat r(0)< r \leq \hat r(\theta)$ and the conclusion follows from~\eqref{e:unif-cont-rtheta}.
	
	\smallskip
	
	\noindent Case 3: If $z_0 = \phi(v_0)$ and $0\leq z < \phi(r\nu_\theta)$, the argument is similar to the previous case. Indeed we have $\epsilon > z_0 - z > \phi(v_0) - \phi(r\nu_\theta)$ and it follows from Lemma~\ref{lem:prop-of-phi} that $r-a_0 \lesssim \sqrt{\epsilon}$. If $r\geq a_0$ this implies~\eqref{e:a-r}. If $r<a_0$, we have $r=\hat r(\theta) < a_0 \leq \hat r(0)$ and the conclusion follows from~\eqref{e:unif-cont-rtheta}.
	
	\smallskip
	
	\noindent Case 4: If $0\leq z_0 < \phi(v_0)$ and $0\leq z < \phi(r\nu_\theta)$ then $a_0=\hat r(0)$, $r=\hat r(\theta)$, and~\eqref{e:a-r} follows from~\eqref{e:unif-cont-rtheta}.
	
	\smallskip
	
	\noindent Case 5: If $0\leq z_0 \leq \phi(v_0)$ and $z<0$, let us first prove that 
	\begin{equation} \label{e:r(theta)-r}
	0\leq \hat r(\theta) - r \lesssim \sqrt{\epsilon} \,\text{ and }\, 0\leq \hat r(0) - a_0 \lesssim \sqrt{\epsilon}.
	\end{equation}
	If $r=\hat r(\theta)$, the first part of~\eqref{e:r(theta)-r} is trivial. If $r<\hat r(\theta)$ then $z=-\phi(-r\nu_\theta)$. Hence $\epsilon > -z+z_0 \geq -z \geq \phi(r \nu_{\theta+\pi}) - \phi(\hat r(\theta+\pi)\nu_{\theta+\pi})$ and Lemma~\ref{lem:prop-of-phi} implies $\hat r(\theta+\pi) - r \lesssim \sqrt{\epsilon}$. Since $K=-K$, we have $\hat r(\theta+\pi) = \hat r(\theta)$ and the first part of~\eqref{e:r(theta)-r} follows. Next, if $0\leq z_0 < \phi(v_0)$ then $a_0=\hat r(0)$ and the second part of~\eqref{e:r(theta)-r} is trivial. If $z_0=\phi(v_0)$ then $\epsilon > z_0 - z \geq \phi(v_0) \geq \phi(v_0) - \phi(\hat r(0) \nu_0)$. Hence Lemma~\ref{lem:prop-of-phi} implies $0\leq \hat r(0) - a_0\lesssim \sqrt{\epsilon}$ and this concludes the proof of~\eqref{e:r(theta)-r}. Finally~\eqref{e:r(theta)-r} together with~\eqref{e:unif-cont-rtheta} give $|a_0-r| \leq |a_0-\hat r(0)|+|\hat r(0)-\hat r(\theta)|+|\hat r(\theta)-r| \lesssim \sqrt{\epsilon}$ which concludes the proof of~\eqref{e:a-r}.	
\end{proof}

Since $K$ is a convex symmetric set that contains $0$ in its interior, there is an absolute constant $\epsilon_1 \in (0,1)$ such that, for every $\theta \in (-\epsilon_1,\epsilon_1)$, we have $\{ s>0 : \, s\nu_\theta - v_0 \in \interior(K)\} \not=\emptyset$ and $l_\theta(s):= s\nu_\theta - v_0 \in \interior (K)$ for every $s\in (0,s_\theta)$ where $s_\theta := \sup \{ s>0 : \,l_\theta(s)  \in \interior(K)\}$.

\medskip

From now on, we let $\epsilon \in (0,\min(\epsilon_0,\epsilon_1))$ be an absolute constant to be chosen small enough later. We let $\theta \in (-\theta_\epsilon,\theta_\epsilon)$ be such that, for a.e.~$s\in (0,s_\theta)$ the following properties hold. First, the maps $s\mapsto \phi(l_\theta(s))$ and $s\mapsto \phi(-l_\theta(s))$ are differentiable at $s$ with derivatives $\langle \nabla \phi(l_\theta(s)),\nu_\theta\rangle$ and $-\langle \nabla \phi(-l_\theta(s)),\nu_\theta\rangle$ respectively. Second, for every $w\in \R^2\setminus\{0\}$, 
\begin{equation} \label{e:cdt-theta-1}
\measuredangle (l_\theta(s),w) \leq \alpha \Rightarrow \langle \nabla\phi(l_\theta(s)) , w \rangle \leq -m \|l_\theta(s)\| \|w\|,
\end{equation}
and
\begin{equation} \label{e:cdt-theta-1-bis}
\measuredangle (-l_\theta(s),w) \leq \alpha \Rightarrow \langle \nabla\phi(-l_\theta(s)) , w \rangle \leq -m \|l_\theta(s)\| \|w\|.
\end{equation}
Third,
\begin{equation} \label{e:cdt-theta-2}
\|l_\theta(s)\|\leq \kappa \Rightarrow \|\nabla\phi(-l_\theta(s))\| \leq M \|l_\theta(s)\|.
\end{equation}
Recall that $\phi \in \Lip_{loc}(\interior(K))$ and we are assuming that~\eqref{e:phi-2} and~\eqref{e:phi-3} hold, hence these properties hold for a.e.~$\theta \in (-\theta_\epsilon,\theta_\epsilon)$. For such a $\theta\in (-\theta_\epsilon,\theta_\epsilon)$, we consider $q=(v,z)\in \partial(\epsilon,\theta,p_0)$ and we will prove that
\begin{equation*} \label{e:conclusion}
p_0^{-1}\cdot \delta_t(q) \in B \text{ and } \delta_t(q)\cdot p_0^{-1} \in B \text{ for every } t\in [0,\epsilon]
\end{equation*}
provided $\epsilon$ is chosen small enough. By Lemma~\ref{lem:a.e.vs-every} and the fact that $B$ is closed, this will complete the proof of Lemma~\ref{lem:sufficient-for-bcp}.

\medskip

We write $v = r\nu_\theta$ for some $r>0$ and we set 
\begin{equation} \label{e:def-gamma}
\gamma(t) := tv - v_0 = rt\nu_\theta - v_0,
\end{equation}
\begin{equation} \label{e:def-k}
k(t) := \|\gamma(t)\|=  \sqrt{ r^2 t^2 - 2a_0rt\cos\theta + a_0^2 }~.
\end{equation}

\medskip

For $t>0$, we have
\begin{equation*}
	p_0^{-1}\cdot \delta_t (q) = \left(\gamma(t), -z_0+t^2z - \frac12 a_0 rt\sin\theta \right)
	\end{equation*}
and 
\begin{equation*}
	\delta_t (q) \cdot p_0^{-1} = \left(\gamma(t), -z_0+t^2z + \frac12 a_0 rt\sin\theta \right).
	\end{equation*}

\medskip

Then $p_0^{-1}\cdot \delta_t (q)\in B$  and $\delta_t (q) \cdot p_0^{-1} \in B$ if and only if
\begin{gather}
\label{e:cdt-1}
	\gamma(t) \in K, \\
\label{e:cdt-2}
	-z_0+t^2z +\frac12 a_0 rt\sin|\theta|
	\le \phi(\gamma(t)), \\
\label{e:cdt-3}
	-\phi(-\gamma(t))
	\le -z_0+t^2z -\frac12 a_0 rt\sin|\theta|.
\end{gather}

\medskip
We first prove that~\eqref{e:cdt-1} holds for every $t\in [0,1]$ provided $\epsilon$ is chosen small enough. We have 
\begin{equation} \label{e:k(1)}
k(1) \lesssim \sqrt{\epsilon}~.
\end{equation}
Indeed, it follows from~\eqref{e:a-r} that 
\begin{equation*} 
\begin{split}
	|r\cos\theta-a_0| & \le a_0(1-\cos\theta) + |a_0-r|\cos\theta \\
			&\le D (1-\cos\epsilon) + |a_0-r| \\
			&\lesssim \sqrt{\epsilon}~.
			 \end{split}
	\end{equation*}
On the other hand, $|r\sin\theta| \le D \sin\epsilon$, and hence~\eqref{e:k(1)} follows. Choosing $\epsilon$ small enough, we get that 
\begin{equation} \label{e:k(1)-kappa}
k(1)\leq \kappa
\end{equation}
where $\kappa$ is given by~\eqref{e:phi-3}. Hence $\gamma(1) \in \interior(K)$. On the other hand, $\gamma(0)=(-a_0,0) \in K$. Since $K$ is convex, we get that $\gamma(t) \in \interior (K)$ for every $t\in (0,1]$. Hence~\eqref{e:cdt-1} holds for every $t\in [0,1]$  as claimed.

\medskip 

We need now to prove that~\eqref{e:cdt-2} and~\eqref{e:cdt-3} hold for every $t \in [0,\epsilon]$ provided $\epsilon$ is chosen small enough.

\medskip

For the proof of~\eqref{e:cdt-2}, we set 
$$f(t):= t^2z +\frac12 a_0 rt\sin|\theta| -z_0 - \phi(\gamma(t)) .$$

\begin{lemma} We have $f(t) \leq 0$, that is,~\eqref{e:cdt-2} holds, for every $t\in [0,\epsilon]$ provided $\epsilon$ is chosen small enough.
\end{lemma}

\begin{proof} First, note that since $k(t)$ is a convex function, it follows from~\eqref{e:k(1)-kappa} that, for every $t\in [0,1]$,
\begin{equation} \label{e:k(t)}
0\leq k(t) \leq \max\{k(0),k(1)\} = \max\{a_0,k(1)\} \leq \max\{a_0,\kappa\}.
\end{equation}
We now divide the proof into two cases.

\smallskip

\noindent Case 1: If $a_0\leq \kappa$ then $k(t)\leq \kappa$ for every $t\in [0,1]$ by~\eqref{e:k(t)} and it follows that $\phi(\gamma(t)) \geq \eta$ where $\eta$ is given by~\eqref{e:def-eta}. Next, it follows from~\eqref{e:phi-2} that, for a.e.~$v\in\interior(K)$, we have $2(\phi(v) - \phi(0) )\leq -m \|v\|^2$. By continuity of $\phi$ on $\interior(K)$ together with~\eqref{e:radial-continuity-phi}, this inequality actually holds true for every $v\in K$. This implies in particular that 
\begin{equation} \label{e:max-phi-at-0}
\phi(v) < \phi(0) \, \text{ for every } v\in K \setminus\{0\}~.
\end{equation}
Hence $z\leq \phi(r\nu_\theta) < \phi(0)$. Since $z_0 \geq 0$, we get for $t\in [0,\epsilon]$
$$f(t) \leq \epsilon^2\phi(0) + \frac{D^2\epsilon}{2} \sin \epsilon - \eta \leq 0$$
provided $\epsilon$ is chosen small enough.

\smallskip

\noindent Case 2: If $a_0\geq \kappa$, then, for $t\in [0,\epsilon]$, we have $rt\leq D \epsilon \leq \kappa / 2 < a_0 \cos \epsilon \leq a_0 \cos \theta \leq a_0$ provided $\epsilon$ is chosen small enough. Hence $0< a_0 \cos \epsilon - tr \leq a_0 \cos \theta - tr \leq a_0 - tr \leq k(t) \leq a_0$. We also have $r/a_0 \leq D/\kappa$. Therefore 
\begin{equation*}
\begin{split}
	\cos \measuredangle (\gamma(t), - \gamma'(t)) &= \frac{\langle \gamma(t),-\gamma'(t) \rangle}{\|\gamma(t)\| \|\gamma'(t)\|} 
	= \frac{a_0\cos\theta - rt }{ k(t) } \\
	&\ge \frac{a_0\cos\epsilon - rt }{ a_0 } \geq \cos\epsilon - \frac{D}{\kappa}\epsilon \geq \cos\alpha
\end{split}
	\end{equation*}
where $\alpha$ is given by~\eqref{e:phi-2}, and where the last inequality holds provided $\epsilon$ is chosen small enough. Hence, by choice of $\theta$ (recall~\eqref{e:cdt-theta-1} and remember that $\gamma(t) = l_\theta(rt)$), we have for a.e.~$t\in [0,\epsilon]$ 
	\[
	\langle \grad\phi(\gamma(t)) , - \gamma'(t) \rangle \le - m r  k(t) .
	\]
Finally, if $\epsilon>0$ is chosen small enough, we get for a.e.~$t\in[0,\epsilon]$
	\begin{align*}
 f'(t) &= 2zt + \frac12 a_0 r\sin|\theta| - \langle \grad\phi(\gamma(t)) , \gamma'(t) \rangle \\
	&\le 2\phi(0) \epsilon + \frac{D^2}{2} \sin\epsilon - m r k(t) \\
	&\le 2\phi(0) \epsilon + \frac{D^2}{2} \sin\epsilon - m r (a_0-tr) \\
	&\le 2\phi(0) \epsilon + \frac{D^2}{2} \sin\epsilon + m D^2 \epsilon - m a_0r.
	\end{align*}
By~\eqref{e:a-r}, we have $r \geq \kappa / 2$ provided $\epsilon$ is chosen small enough, and we get for a.e.~$t\in [0,\epsilon]$ $$f'(t) \leq 2\phi(0) \epsilon + \frac{D^2}{2} \sin\epsilon + m D^2 \epsilon - \frac{m \kappa^2}{2} \leq 0$$
provided $\epsilon$ is chosen even smaller if necessary. Since $ f(0) = -z_0-\phi(-a,0)\le 0$, it follows that $f(t)\le 0$ for every $t\in[0,\epsilon]$.
\end{proof}

\medskip

To  prove that~\eqref{e:cdt-3} holds for every $t\in [0,\epsilon]$ provided $\epsilon$ is chosen small enough, we set $$g(t) := \phi(-\gamma(t)) + t^2z - \frac12 a_0 rt \sin|\theta| - z_0 .$$ We have $g(0)=\phi(a_0 \nu_0)-z_0\ge0$. We show in Lemma~\ref{lem:g-1} and Lemma~\ref{lem:g-2} below that $g'(t) \geq 0$ for a.e.~$t\in [0,\epsilon]$ provided $\epsilon$ is chosen small enough. It will follow that $g(t) \geq 0$, that is,~\eqref{e:cdt-3} holds, for every $t\in [0,\epsilon]$ as required.

\medskip

We remind  that $\epsilon >0$ has already be chosen small enough so that in particular $\epsilon < \alpha$, and hence $|\theta| < \alpha$,  where $\alpha$ is given by~\eqref{e:phi-2}. We set 
\begin{equation} \label{e:def-t^-}
t^- := \inf \{ t\geq 0: \measuredangle (-\gamma(t),\gamma'(t)) > \alpha \}.
\end{equation}

If $\theta = 0$, $\measuredangle (-\gamma(t),\gamma'(t)) = 0$ if $0\leq t<a_0/r$ and $\measuredangle (-\gamma(t),\gamma'(t)) = \pi$ if $t>a_0/r$, hence $t^-=a_0/r$.

If $\theta \not = 0$, then $\gamma(t) \not = 0$ for every $t\geq 0$ and 
$$\cos \measuredangle (-\gamma(t),\gamma'(t)) = \frac{ \langle -\gamma(t), \gamma'(t) \rangle }{\|\gamma(t)\| \|\gamma'(t)\|} = \frac{a_0\cos\theta - tr }{ k(t) }.$$
For $t=0$, we have $\cos \measuredangle (-\gamma(0),\gamma'(0)) = \cos \theta > \cos \alpha$. When $t\rightarrow +\infty$, $\cos \measuredangle (-\gamma(t),\gamma'(t)) \rightarrow -1$. Hence $t^- \in (0, +\infty)$ with 
$$t^- = \min\left\{t > 0: \frac{a_0\cos\theta - tr }{ k(t) } \leq \cos \alpha\right\} .$$  
Moreover, a routine computation shows that 
\begin{equation} \label{e:t^-}
t^- = \frac{a_0}{r} \left( \cos\theta - \frac{\cos\alpha}{\sin\alpha} \sin|\theta| \right)
\end{equation}
provided $\epsilon$ is chosen small enough. Note that this formula allows to recover the value of $t^-$ when $\theta = 0$.

Furthermore, choosing $\epsilon$ even smaller if necessary and in particular so that $\cos\epsilon - \cos \alpha \sin\epsilon /\sin\alpha \geq 1/2$, we get
\begin{equation} \label{e:t^-lowerbound}
t^- \geq \frac{a_0}{r} \left( \cos\epsilon - \frac{\cos\alpha}{\sin\alpha} \sin \epsilon \right)\geq \frac{a_0}{2r}.
\end{equation}

\begin{lemma}\label{lem:g-1}
We have $g'(t)\ge0$ for a.e.~$t \in [0, \min(\epsilon,t^-/2)]$ provided $\epsilon$ is chosen small enough.
\end{lemma}

\begin{proof}
For a.e.~$t\in [0,\min(\epsilon,t^-/2)]$, we have
\begin{equation*}
g'(t) = -\langle \grad\phi(-\gamma(t)) ,\gamma'(t) \rangle + 2tz - \frac12 a_0r\sin|\theta|.
\end{equation*}
By definition of $t^-$, see~\eqref{e:def-t^-}, together with~\eqref{e:cdt-theta-1-bis} (remember that $\gamma(t) = l_\theta(rt)$), we get for a.e.~$t\in [0,\min(\epsilon,t^-/2)]$
$$-\langle \grad\phi(-\gamma(t)) ,\gamma'(t) \rangle \geq m r k(t) \ge ma_0r \left(\cos\theta -  \frac{rt}{a_0} \right)$$
where the last inequality follows from the fact that $(a_0\cos\theta-tr) k(t)^{-1} = \cos \measuredangle (-\gamma(t),\gamma'(t)) \leq 1$. By~\eqref{e:t^-} it follows that 
\begin{equation*}
\begin{split}
-\langle \grad\phi(-\gamma(t)) ,\gamma'(t) \rangle &\geq  ma_0r \left(\cos\theta -  \frac{r}{2a_0}\,t^- \right) \\
&\geq \frac{ma_0r}{2} \left(\cos \theta + \frac{\cos\alpha}{\sin\alpha}  \sin |\theta|\right) \\
&\geq \frac{ma_0r}{2} \cos\epsilon \\
&\geq \frac{ma_0r}{4}
\end{split}
\end{equation*}
provided $\epsilon$ is chosen small enough, and hence,
$$g'(t) \geq a_0r \left(\frac{m}{4} - \frac{\sin \epsilon}{2}\right) + 2tz \geq \frac{ma_0r}{8} + 2tz$$
provided $\epsilon$ is chosen even smaller if necessary. If $z\geq 0$, it follows that $g'(t) \geq 0$. If $z<0$, by~\eqref{e:r(theta)-r} we have $0\leq \hat r(\theta) - r \lesssim \sqrt{\epsilon}$ and $0\leq \hat r(0) - a_0 \lesssim \sqrt{\epsilon}$. Since $\hat r(0)\geq \kappa$ and $\hat r(\theta) \geq \kappa$, we get $a_0 \geq \kappa/2$ and $r\geq \kappa / 2$ provided $\epsilon$ is chosen small enough. Moreover, we have $z \geq -\phi(-v) > -\phi(0)$ by~\eqref{e:max-phi-at-0}. Hence,
$$g'(t) \geq \frac{m\kappa^2}{32} - 2\phi(0) \epsilon \geq 0$$
for a.e.~$t\in [0,\min(\epsilon,t^-/2)]$, provided $\epsilon$ is chosen small enough.
\end{proof}

\begin{lemma} \label{lem:g-2} We have $g'(t)\geq 0$ for a.e.~$t \in [ \min(\epsilon,t^-/2),\epsilon]$ provided $\epsilon$ is chosen small enough.
\end{lemma}

\begin{proof} If $t^- \geq 2\epsilon$, the lemma is trivial. If $t^-  < 2\epsilon$, by~\eqref{e:t^-lowerbound} we have $a_0 \leq 2rt^- \leq 4D\epsilon$. Then, if we choose $\epsilon$ small enough,~\eqref{e:r(theta)-r} together with the fact that $\hat r(0) \geq \kappa$ implies $z\geq 0$. On the other hand,~\eqref{e:a-r} implies $r\lesssim \sqrt{\epsilon} \leq \kappa$. It follows that $z=\phi(r\nu_\theta)\ge\eta$ where $\eta$ is given by~\eqref{e:def-eta}. Furthermore, choosing $\epsilon$ even smaller if necessary, by~\eqref{e:k(t)} we have $k(t) \leq \kappa$ for every $t\in [0,1]$. Remembering~\eqref{e:cdt-theta-2} and the fact that $\gamma(t) = l_\theta (rt)$, it follows that for a.e.~$t\in[0,1]$
\begin{equation} \label{e:g'-lowerbound}
\begin{split}
	g'(t) &= -\langle \grad\phi(-\gamma(t)) ,\gamma'(t) \rangle + 2tz - \frac12 a_0r\sin|\theta| \\
	&\ge -M\|\gamma(t)\|\|\gamma'(t)\| + 2t\eta - \frac12 a_0r \sin\epsilon \\
	&=  2\eta t - M k(t) r - \frac12 a_0r \sin\epsilon .
\end{split}
\end{equation}
Define $$h(t):=2\eta t -M k(t) r - \frac12 a_0 r \sin\epsilon$$ and set $\beta_\theta:=\cos\theta-\cos\alpha\sin|\theta|/\sin\alpha >0$. Recalling~\eqref{e:t^-}, we have 
	\begin{align*}
	h\left(\frac{t^-}{2}\right) &= h\left(\frac{a_0\beta_\theta}{2r}\right)\\
	&= \frac{a_0\beta_\theta \eta}{r} -Mr \sqrt{a_0^2 + a_0^2 \frac{\beta_\theta^2}{4} - a_0^2 \beta_\theta \cos\theta}  - \frac12 a_0r \sin\epsilon\\
	&= a_0r \left(\frac{\beta_\theta \eta}{r^2} - M \sqrt{1+\frac{\beta_\theta^2}{4} - \beta_\theta \cos\theta}- \frac12 \sin\epsilon\right) \\
	& \geq a_0r \left(\frac{\beta_\theta \eta}{r^2} - M \sqrt{1+\frac{\beta_\theta^2}{4}} - \frac12 \sin\epsilon\right).
\end{align*}
Since $|\theta|< \epsilon$, we have $0<\beta_\epsilon < \beta_\theta \leq 1$, and, recalling that $r\leq C\sqrt{\epsilon}$ for some absolute constant $C>0$, we get
\begin{equation*}
	h\left(\frac{t^-}{2}\right)\geq a_0r \left(\frac{\beta_\epsilon \eta}{C^2 \epsilon}  - \sqrt{\frac{5}{4}} M - \frac12 \sin\epsilon\right) .
	\end{equation*}
It follows that 
\begin{equation} \label{e:h(t^-/2)} 
h\left(\frac{t^-}{2}\right)\geq 0
\end{equation}
provided $\epsilon$ is chosen small enough.

If $\theta \not= 0$, then $k(t) \not= 0$ for every $t\geq 0$ and $h\in C^1([0,+\infty))$ with
\begin{equation*}
h'(t) = 2\eta -M r^2 (rt-a_0\cos\theta) k(t)^{-1}.
\end{equation*}
We have $h'(t) \not=0$ for every $t\geq 0$. Indeed, arguing by contradiction, assume that $h'(t)=0$ for some $t\geq 0$. Then $t$ is a real solution of the quadratic equation
\[
	t^2 r^2 (4\eta^2-M^2r^4) + 2 t a_0r\cos\theta ( M^2r^4-4\eta^2) + a_0^2 (4\eta^2-M^2r^4\cos^2\theta) = 0 .
\]
Note that $4\eta^2-M^2r^4 \geq 4\eta^2 -M^2 C^4 \epsilon^2 >0$ provided we choose $\epsilon$ small enough.
Thus, the discriminant of this equation is given by
\[
	\Delta = - 16\eta^2a_0^2r^2 (4\eta^2-M^2r^4)  (1-\cos^2\theta) <0 .
\]
Hence the above quadratic equation does not have real solutions. This gives a contradiction and  proves that $h'(t) \not=0$ for every $t\geq 0$. Since $h'(0) = 2\eta + Mr^2\cos\theta \geq 2\eta >0$, it follows that $h'(t) > 0$ for every $t\geq 0$. Together with~\eqref{e:h(t^-/2)}, this implies that $h(t) \geq 0$ for every $t \geq t^- / 2$. Then~\eqref{e:g'-lowerbound} implies $g'(t) \geq h(t) \geq 0$ for a.e.~$t\in [t^-/2,\epsilon]$ and this concludes the proof of the lemma when $\theta\not=0$.

If $\theta = 0$ , then $k(t) = |rt-a_0|$ and we have for $t\not= a_0/r$,
\[
	h'(t) = 2\eta - Mr^2 {\rm sgn}(rt-a_0) \ge 2\eta - Mr^2 \geq 2\eta -M C^2 \epsilon >0
	\]
	if $\epsilon$ is chosen small enough, and we conclude the proof as in the case where $\theta\not=0$.
\end{proof}

\section{Necessary conditions} \label{sect:necessary-conditions}

This section is devoted to the proof of necessary conditions for the validity of BCP on $\HH$ equipped with a homogeneous distance. In Section~\ref{subsect:preliminaries-nec-cdt}, we recall a necessary condition proved in~\cite{LeDonne_Rigot_Heisenberg_BCP} and give a consequence of this condition on which our arguments will be mainly based, see Proposition~\ref{prop:necessary-condition-bcp}. In Section~\ref{subsect:nec-cdt-general-case}, we prove Theorem~\ref{thm:intro-necessary-cdt}, see Theorem~\ref{thm:necessary-cdt-int(K)}, Theorem~\ref{thm:radially-decreasing-profile}, and Theorem~\ref{thm:diff-at-0} for more detailed statements. In Section~\ref{subsect:additional-nec-cdt}, we prove several improvements of the necessary conditions given in Section~\ref{subsect:nec-cdt-general-case} under additional assumptions about the regularity of the profile of the homogeneous distance. Theorem~\ref{thm:intro-more-nec-cdt} will come as a corollary of these improvements.
 
\subsection{Preliminaries} \label{subsect:preliminaries-nec-cdt}

For technical convenience, we identify when needed in this section $\HH$ with $\R^3$ in the obvious way. Namely, we identify $p=(v,z)\in\HH$, where $v=(x,y)\in\R^2$ and $z\in\R$, with $(x,y,z) \in \R^3$. 

\medskip

For $q\in \HH$, $u \in\R^3\setminus\{0\}$ and $\alpha \in (0,\pi/2)$, we denote by $\mathcal{C}(q,u,\alpha)$ the open Euclidean half-cone with vertex $q$, axis $q+\R_{>0}\, u$ and opening $2\alpha$, namely, 
$$\mathcal{C}(q,u,\alpha):=\{p\in \HH\setminus\{q\}:\, \measuredangle(u,p-q) < \alpha\}~.$$

\smallskip

For $q\in \HH$, we denote by $\tau_q$ the left-translation $\tau_q(p) := q\cdot p$. When considering $(\tau_{q})_*(v)$ for $v\in \R^2$, we identify $v=(v_1,v_2)$ with $(v_1,v_2,0)\in\R^3$, namely, $(\tau_{q})_*(v) := (v, \omega(\pi(q),v)/2)\in\R^3$ (see Section~\ref{sect:preliminaries} for the definition of $\pi$). 

\medskip

Given a homogeneous distance $d$ on $\HH$ with unit ball centered at the origin~$B$, we say that $u\in \R^3\setminus\{0\}$ points out of $B$ at $q\in\partial B$ if there exist an open neighbourhood $U$ of $q$ and $\alpha \in (0,\pi/2)$ such that 
$$B \cap \mathcal{C}(q,u,\alpha) \cap U = \emptyset~.$$
For $v\in\R^2\setminus\{0\}$ and $\epsilon>0$, we denote by $\Omega_\epsilon( v)$ the set of points $q\in\de B$ such that $\pi(q)\not= 0$, $\measuredangle(v,\pi(q))< \epsilon$ and $(\tau_{q})_*(-v)$ points out of $B$ at $q$.

\medskip

The following sufficient condition for the non-validity of BCP can be found in~\cite{LeDonne_Rigot_Heisenberg_BCP}.

\begin{theorem} \cite[Theorem 6.1]{LeDonne_Rigot_Heisenberg_BCP} \label{thm:criterum-for-bcp}
Let $d$ be a homogeneous distance on $\HH$.	Assume that one can find $v\in\R^2\setminus\{0\}$ such that $\Omega_\epsilon(v)\neq\emptyset$ for every $\epsilon>0$. Then BCP does not hold in $(\HH,d)$.
\end{theorem}

Note that $\Omega_{\epsilon'}(v) \subset \Omega_\epsilon(v)$ whenever $0< \epsilon' \leq  \epsilon$. Hence $\Omega_\epsilon(v) = \emptyset$ implies $\Omega_{\epsilon'}(v) = \emptyset$ for every $0< \epsilon' \leq  \epsilon$. Therefore, Theorem~\ref{thm:criterum-for-bcp} can be rephrased as a necessary condition for the validity of BCP in the following way.

\begin{corollary}  \label{cor:necessary-condition-bcp}
Let $d$ be a homogeneous distance on $\HH$ and assume that BCP holds in $(\HH,d)$. Then, for all $v\in\R^2\setminus\{0\}$, there exists $\epsilon_v >0$ such that $\Omega_{\epsilon}(v) = \emptyset$ for every $0<\epsilon \leq \epsilon_v$.
\end{corollary} 

The major part of our arguments in Sections~\ref{subsect:nec-cdt-general-case} and~\ref{subsect:additional-nec-cdt} will be based on the next proposition which is a consequence of Corollary~\ref{cor:necessary-condition-bcp}.

\begin{proposition} \label{prop:necessary-condition-bcp}
Let $d$ be a homogeneous distance on $\HH$ and assume that BCP holds in $(\HH,d)$. Then, for all $v\in\R^2\setminus\{0\}$, there exists $\epsilon_v \in (0,\pi/4)$ such that the following holds. Let $q\in \partial B$ be such that $\pi(q)\not= 0$ and $\measuredangle(v,\pi(q))< \epsilon_v$. Assume that there is an open neighbourhood $U$ of $q$ and a continuous map $F:U\rightarrow \R$ which is differentiable at $q$ and such that $F(q)=0$ and $B\cap U = \{p\in U:\, F(p) \leq 0\}$. Then $\langle \grad F(q),(\tau_q)_*(v) \rangle \geq 0$.
\end{proposition}

\begin{proof} Let $v\in\R^2\setminus\{0\}$ and $\epsilon_v \in (0,\pi/4)$ be given by Corollary~\ref{cor:necessary-condition-bcp} so that $\Omega_{\epsilon_v}(v) = \emptyset$. Let $q\in \partial B$ be as in the statement of the proposition, and in particular, such that $\pi(q)\not= 0$ and $\measuredangle(v,\pi(q))< \epsilon_v$. Since $\Omega_{\epsilon_v}(v) = \emptyset$, $(\tau_{q})_*(-v)$ does not point out of $B$ at $q$. Hence, for every $k\in \N^*$, one can find $q_k \in B \cap \mathcal{C}(q,(\tau_{q})_*(-v),1/k) \cap U$. It follows that $F(q_k) - F(q) = F(q_k) \leq 0$. Since $q_k \in \mathcal{C}(q,(\tau_{q})_*(-v),1/k)$, we have $$\frac{q_k-q}{\|q_k-q\|} \underset{k\rightarrow +\infty}{\longrightarrow} \frac{(\tau_{q})_*(-v)}{\|(\tau_{q})_*(-v)\|}.$$ Since $F$ is differentiable at $q$, it follows that $$\frac{F(q_k) - F(q)}{\|q_k-q\|} \underset{k\rightarrow +\infty}{\longrightarrow}  \langle \nabla F(q), \frac{(\tau_{q})_*(-v)}{\|(\tau_{q})_*(-v)\|} \rangle$$
and hence $\langle \nabla F(q), (\tau_{q})_*(v) \geq 0$.\end{proof}

\subsection{Necessary conditions} \label{subsect:nec-cdt-general-case}
In this section, we let $d$ denote a homogeneous distance on $\HH$ with unit ball centered at the origin $B$ and profile $\phi:K \rightarrow [0,+\infty)$ (see Proposition~\ref{prop:LDNG}). We also set 
\begin{equation} \label{e:def-F}
F_+(w,z) := z - \phi(w)\, \text{ and }\, F_-(w,z) := -z -\phi(-w)
\end{equation}
for $w\in K$ and $z\in \R$. In other words, we have
\begin{equation*}
\begin{split}
B = \{(& w,z)\in\HH:\, w\in K,\, F_+(w,z)\leq 0 \} \\ 
&\cap \, \{(w,z)\in\HH:\, w\in K,\, F_-(w,z)\leq 0 \}.
\end{split}
\end{equation*}

We begin with a simple consequence of Proposition~\ref{prop:necessary-condition-bcp} for later use.

\begin{lemma} \label{lem:first-nec-cdt}
	Assume that BCP holds in $(\HH,d)$.
	Then, for all $v\in\bb S^1$, there exists $\epsilon_v \in (0,\pi/4)$ such that the following holds.
	If $w\in \interior(K)\setminus\{0\}$ is such that $\measuredangle(v,w)<\epsilon_v$ and $\phi$ is differentiable at $w$, then 
		\begin{equation}
		\label{e:nec-cdt-1}
			\langle \grad\phi(w), v \rangle +\frac12\omega(v,w) \le 0  
			\end{equation}
and
\begin{equation}
		\label{e:nec-cdt-2}
			 \langle \grad\phi(w), v \rangle -\frac12\omega(v,w) \le 0 .
		\end{equation}
\end{lemma}
 
\begin{proof} Given $v\in \bb S^1$, let $\epsilon_v \in (0,\pi/4)$ be given by Proposition~\ref{prop:necessary-condition-bcp}. Let $w\in \interior(K)\setminus\{0\}$ be such that $\measuredangle(v,w)<\epsilon_v$ and assume that $\phi$ is differentiable at $w$. Set $q_+:=(w,\phi(w)) \in \partial B$. By Proposition~\ref{prop:necessary-condition-bcp}, we have
 $$ \langle \grad F_+(q_+), (\tau_{q_+})_*(v)\rangle = -\langle \grad\phi(w) , v \rangle +\frac12\omega(w,v) \geq 0$$ 
	and~\eqref{e:nec-cdt-1} follows. Similarly, set $q_-:=q_+^{-1}=(-w,-\phi(w))\in \partial B$. By Proposition~\ref{prop:necessary-condition-bcp}, we have $$ \langle \grad F_-(q_-), (\tau_{q_-})_*(-v)\rangle = - \langle \grad\phi(w) , v \rangle -\frac12\omega(w,v) \geq 0$$ 
and~\eqref{e:nec-cdt-2} follows.	
\end{proof}

The first part of Theorem~\ref{thm:intro-necessary-cdt} can deduced from the next proposition, see also Theorem~\ref{thm:necessary-cdt-int(K)}.

\begin{proposition} \label{prop:necessary-cdt-int(K)-bis}
	Assume that BCP holds in $(\bb H,d)$. 
	Then there is a countable closed set $W\subset\bb S^1$ such that, 
	for all $u\in\bb S^1\setminus W$, there are $m_{u}>0$, $\alpha_{u}\in(0,\pi/2)$, and an open subset $O_{u}\subset\bb S^1$ containing $u$ such that
	\begin{equation} 
		\langle \grad\phi(w) , w' \rangle  \le -  m_{u} \|w \|\|w'\| 
	\end{equation} 
	for all $w\in \interior(K) \setminus \{0\}$ such that $\phi$ is differentiable at $w$ and $ w/\|w\| \in O_{u}$, and for all $w'\in\R^2\setminus\{0\}$ such that $\measuredangle(w,w') < \alpha_{u}$.
\end{proposition}

\begin{proof}
	For $v\in \bb S^1$, let $\epsilon_v \in (0,\pi/4)$ be given by Lemma~\ref{lem:first-nec-cdt} and set
	\begin{align*}
	\Gamma_v^+ &:= \{w\in \bb S^1 \setminus\{-v\}:\, \angle (v,w) \in (0, \epsilon_v)\}~,\\
	\Gamma_v^- &:= \{w\in \bb S^1 \setminus\{-v\}:\, \angle (v,w)\in (-\epsilon_v,0) \}~,
	\end{align*}
	and
	\begin{align*}
	W^+ &:= \bb S^1 \setminus \bigcup_{v\in \bb S^1} \Gamma_v^+~,\\
	W^- &:= \bb S^1 \setminus \bigcup_{v\in \bb S^1} \Gamma_v^-~.
	\end{align*}

Then $W:=W^+\cup W^-$ is clearly closed. To prove that $W$ is countable, we first claim that $\Gamma_w^+\cap \Gamma_{w'}^+=\emptyset$ whenever $w$ and $w'$ are distinct points in $W^+$. Indeed assume that $w$ and $w'$ are distinct points in $\bb S^1$ such that $\Gamma_w^+\cap \Gamma_{w'}^+\not=\emptyset$. Then $w\not= -w'$ and we can assume with no loss of generality that $\angle (w,w') >0$. Let $u\in \Gamma_w^+\cap \Gamma_{w'}^+$. We have $\angle (w',u) >0$. This implies that $0<\angle (w,w') < \angle(w,u) < \epsilon_w$ and hence $w'\in \Gamma_w^+$. It follows that $w'\not \in W^+$ which proves the claim. Thus $\{\Gamma_w^+\}_{w\in W^+}$ is a family of disjoint open subsets of $\bb S^1$ and hence is countable. One shows in a similar way that $W^-$ is  countable.
	

Let $u\in\bb S^1\setminus W$.
	Then there are $v,v'\in\bb S^1$ such that $u\in\Gamma_v^+\cap\Gamma_{v'}^-$. Note that since $0 < \epsilon_v , \epsilon_{v'}  < \pi/4$, we have $\angle (v,v') \in (0,\pi/2)$. In particular, $v'\not\in \R v$. Since $\Gamma_v^+ \cap \Gamma_{v'}^-$ is an open subset of $\bb S^1$, one can find $\delta_{u} \in (0,\pi/4)$ such that $w \in \Gamma_v^+ \cap \Gamma_{v'}^-$ for all $w\in \bb S^1$ such that $\measuredangle(u,w) < \delta_{u}$.
	Define $\alpha_u:=\delta_u / 2$ and 
	\[
	O_{u}: = \{w\in\bb S^1: \measuredangle(u,w)<\alpha_u\}.
	\]
Note that if $w\in O_{u}$ and $w' \in \bb S^1$ are such that $\measuredangle(w,w')<\alpha_u$, then $\measuredangle( u,w')<  \delta_{u} $, and hence, $w'\in\Gamma_v^+ \cap \Gamma_{v'}^-$.

	\begin{figure}[ht]
	\includegraphics[width=0.7\textwidth]{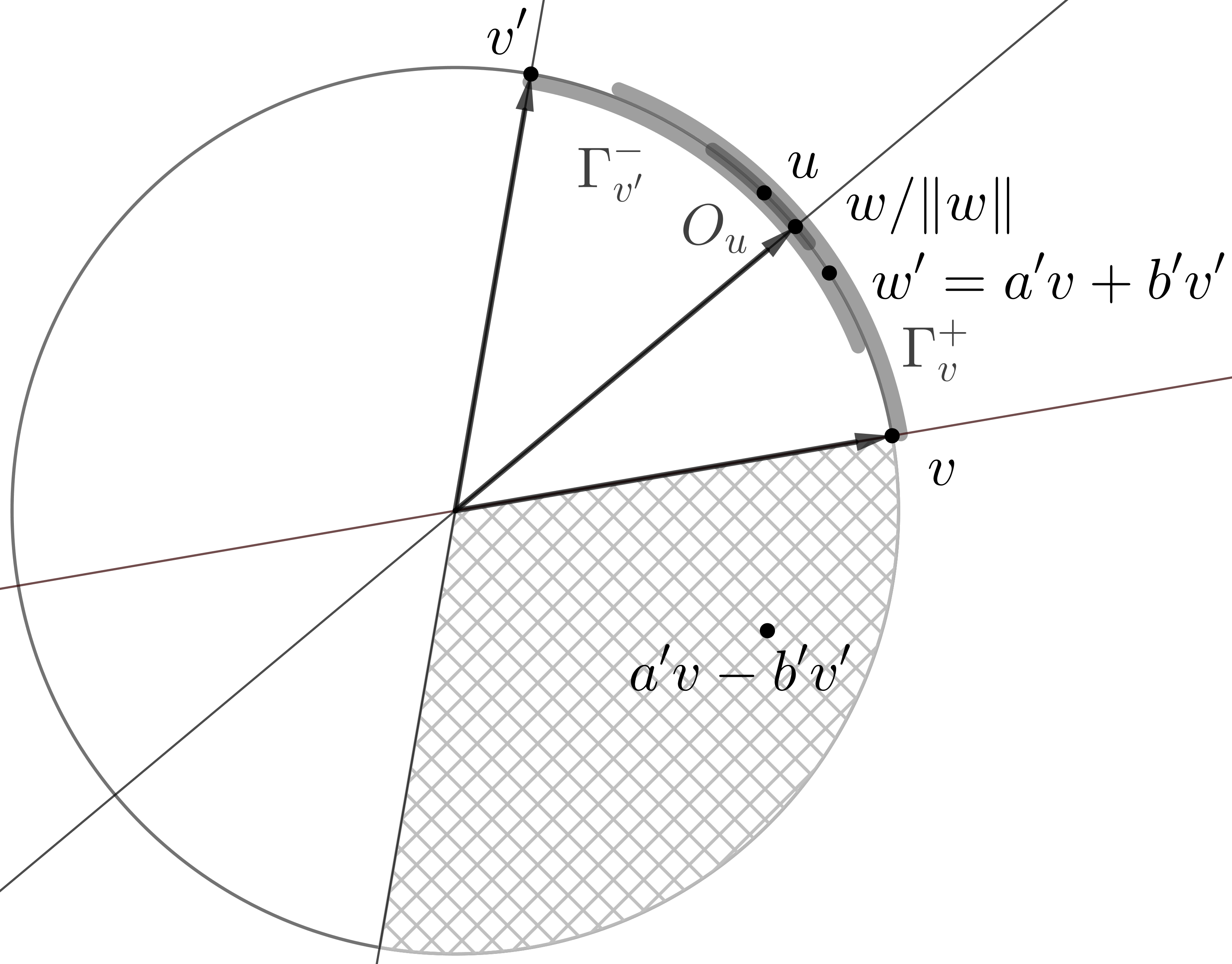}
	\caption{$\angle(a'v-b'v' , w) \in \left(\alpha_u , \pi-\alpha_u\right)$} \label{fig:reflection}
	\end{figure}

	Let $w\in \interior (K) \setminus\{0\}$ be such that $\phi$ is differentiable at $w$ and $w / \|w\| \in O_{u}$, and  let $w'\in\bb S^1$ be such that $\measuredangle(w,w')<\alpha_u$. By the previous remark, we have $w'=a'v+b'v'$ for some $a'>0$ and $b'>0$. Then 
	\[
	\langle \grad\phi(w),w' \rangle
	= a'\langle \grad\phi(w),v \rangle + b' \langle \grad\phi(w),v' \rangle.
	\]
Using \eqref{e:nec-cdt-1} and \eqref{e:nec-cdt-2}, we obtain
	\[
	\langle \grad\phi(w),w' \rangle
	\le -\frac12 a'\omega(v,w) + \frac12 b'\omega(v',w)
	= -\frac12 \omega(a'v-b'v',w)
	\]
	and
	\[
	\langle \grad\phi(w),w' \rangle
	\le \frac12 a'\omega(v,w) - \frac12 b'\omega(v',w)
	= \frac12 \omega(a'v-b'v',w) ,
	\]
	that is,
	\begin{equation*}
	\langle \grad\phi(w),w' \rangle
	\le - \frac{1}{2}\left|\omega(a'v-b'v',w)\right|.
	\end{equation*}
Since $w/\|w\| \in 	O_{u}$, we have $\alpha_u < \angle (v,w) < \epsilon_v$ and $-\epsilon_{v'} < \angle (v',w) < -\alpha_u$. Hence, since $a'>0$ and $b'>0$, we get
\begin{equation*} 
	\angle(a'v-b'v' , w) \in \left(\alpha_u , \pi-\alpha_u\right) ,
\end{equation*}
see Figure~\ref{fig:reflection}.
	Moreover, 
	\[
	C :=\inf\{\|a'v-b'v'\|:a'v+b'v'\in\bb S^1\}>0 ,
	\]
	because the map $(a',b')\in\R^2\mapsto a'v-b'v'\in\R^2$ is a linear isomorphism.
	Therefore,
	\begin{equation*}
	\begin{split}
	\langle \grad\phi(w),w' \rangle
	&\le - \frac12 \left| \omega(a'v-b'v',w)\right|\\
	&= -\frac12 \, \|a'v-b'v'\| \, \|w\| \, |\sin\angle(a'v-b'v',w)| \\
	&\le -m_u \|w\|
	\end{split}
	\end{equation*}
	with $m_u:=C \sin (\alpha_u) / 2$.
\end{proof}

We gather in Theorem~\ref{thm:necessary-cdt-int(K)} below the necessary conditions that can be deduced from Lemma~\ref{lem:first-nec-cdt} and Proposition~\ref{prop:necessary-cdt-int(K)-bis}.

\begin{theorem} \label{thm:necessary-cdt-int(K)}
	Assume that BCP holds in $(\bb H,d)$. 
	Then 
	\begin{equation} \label{e:necessary-cdt-gradphi-int(K)-1}
	\langle \grad \phi(w) , w \rangle \leq 0 \, \text{ for all }w\in \interior(K) \setminus \{0\} 
	\text{ where } \phi \text{ is differentiable}.
	\end{equation}
	
	Moreover, there is a countable closed set $W\subset\bb S^1$ such that $\bb S^1 \setminus W$ can be written as a countable union of relatively open subsets $O_j$ and, for each $j$, there are $m_j >0$ and $\alpha_j\in (0,\pi/2)$ such that
		\begin{equation} \label{e:necessary-cdt-gradphi-int(K)-2}
		\langle \grad\phi(w) , w' \rangle  \le -  m_j \|w \|\|w'\| 
		\end{equation} 
	for all $w\in \interior(K) \setminus \{0\}$ such that $\phi$ is differentiable at $w$ and $ w/\|w\| \in O_j$, and for all $w'\in\R^2\setminus\{0\}$ such that $\measuredangle(w,w')\le\alpha_j$.

	In particular,
	\begin{equation} \label{e:necessary-cdt-gradphi-int(K)-2a.e.}
	\langle \grad\phi(w) , w\rangle < 0 \, \, \text{ for a.e.~} w\in \interior(K) \setminus \{0\}~.
	\end{equation}
\end{theorem}

\begin{proof} The first claim~\eqref{e:necessary-cdt-gradphi-int(K)-1} is a straightforward consequence of Lemma~\ref{lem:first-nec-cdt} when applied with $v:=w/\|w\|$ for $w\in \interior(K) \setminus \{0\}$ such that $\phi$ is differentiable at $w$.
The next two claims~\eqref{e:necessary-cdt-gradphi-int(K)-2} and~\eqref{e:necessary-cdt-gradphi-int(K)-2a.e.} can be deduced from Proposition~\ref{prop:necessary-cdt-int(K)-bis} by standard arguments.
\end{proof}

\smallskip

For rotationally invariant homogeneous distances, Theorem~\ref{thm:necessary-cdt-int(K)} has the following consequence, Corollary~\ref{cor:necessary-rot-inv} below. We refer to the paragraph before Theorem~\ref{thm:intro-characterization-rot-inv-dist} for the definition of rotationally invariant distances and of their radial profile.

\begin{corollary} \label{cor:necessary-rot-inv}
Let $d$ be a rotationally invariant homogeneous distance on $\HH$ with radial profile $\varphi:[-r_d,r_d] \rightarrow [0,+\infty)$. Assume that BCP holds in $(\HH,d)$. Then there is $m>0$ such that $\varphi'(s) \leq -m s$ for a.e.~$s\in (0,r_d)$.
\end{corollary} 

\begin{proof}
Let $\phi$ denote the profile of $d$ so that $\phi(v) = \varphi(\|v\|)$ for $v\in K:=\{ \|v\| \in \R^2 : \|v\| \leq r_d\}$. Since $\phi\in \Lip_{loc}(\interior(K))$ and $\varphi \in \Lip_{loc}((-r_d,r_d))$, for a.e.~$v\in \bb S^1$, we have
\begin{equation*}
\langle \grad\phi(sv) , v \rangle = \varphi'(s) \quad \text{for a.e. } s \in (0,r_d).
\end{equation*}
By~\eqref{e:necessary-cdt-gradphi-int(K)-2}, one can find such a $v\in \mathbb{S}^1$ such that there is $m>0$ such that $ \langle \nabla \phi (sv) , v \rangle \leq -m s$ for a.e.~$s\in (0,r_d)$ which gives the conclusion.
\end{proof}

The "only if" part in Theorem~\ref{thm:intro-characterization-rot-inv-dist} is a consequence of Corollary~\ref{cor:necessary-rot-inv}, which can for instance be used to recover the fact that the Kor\'anyi distance~\eqref{e:def-koranyi-dist} does not satisfy BCP. Indeed, the radial profile of the Kor\'anyi distance is given by $$\varphi(s) :=  \frac{1}{4} \left(1-s^4\right)^{1/2}$$ for $s\in [-1,1]$ and hence $$\sup_{s\in (0,1)} \frac{\varphi'(s)}{s} = \sup_{s\in (0,1)} \frac{-s^2}{2\left(1-s^4\right)^{1/2}} = 0.$$

\smallskip

\begin{remark} \label{rmk:nobcp-cc-box-dist} As a consequence of Corollary~\ref{cor:necessary-rot-inv}, we also explicitly note that if $d$ is a rotationally invariant homogeneous distance on $\HH$ that satisfies BCP, then its radial profile $\varphi:[-r_d,r_d] \rightarrow [0,+\infty)$ is decreasing on $[0,r_d]$. As a straightforward application, one recovers the fact, already mentioned in~\cite[p.1608]{LeDonne_Rigot_Heisenberg_BCP}, that rotationally invariant homogeneous distances on $\HH$ with radial profile $\varphi:[-r_d,r_d] \rightarrow [0,+\infty)$ such that $\varphi(s) = \max_{[0,r_d]} \varphi$ for some $s\in (0,r_d]$ do not satisfy BCP. We recall that examples of such distances are given by the Carnot-Carath\'eodory distance and the rotationally invariant homogeneous distances $d_\infty$ and $\rho_\infty$ given by 
\begin{equation*}
d_\infty(0,(v,z)):= \max \{ \|v\|, 2|z|^{1/2} \} 
\end{equation*}
and
\begin{equation*}
\rho_\infty(0,(v,z)):= \max \{ |x|, |y|, |2z|^{1/2} \}
\end{equation*}
where $v:=(x,y)$.
\end{remark}

\smallskip

We go back now to general homogeneous distances $d$ with profile $\phi: K \rightarrow [0,+\infty)$. As a simple consequence of Theorem~\ref{thm:necessary-cdt-int(K)} together with the fact that $\phi$ is locally Lipschitz on $\interior (K)$, we get that $\phi$ achieves its maximum on $K$ at the origin. 

\begin{theorem} \label{thm:radially-decreasing-profile}
Assume that BCP holds in $(\HH,d)$. Then, for every $w\in \partial K$, the map $t \mapsto \phi(tw)$ is non increasing on $[0,1]$, and hence $\phi(0) = \max\{\phi(w):\, w\in K\}$.
\end{theorem}

\begin{proof}
For $v\in \bb S^1$, set $t_v := \max\{t>0:\; tv \in K\}$ and $\phi_v(t) := \phi(tv)$ for $t\in [0,t_v]$. Since $\phi\in \Lip_{loc}(\interior(K))$ is differentiable a.e.~on $\interior(K)$, for a.e.~$v\in \bb S^1$, $\phi_v$ is differentiable a.e.~on $(0,t_v)$ with $$\phi_v'(t) =  t^{-1} \langle \grad\phi(tv) , tv \rangle  \leq 0$$ by \eqref{e:necessary-cdt-gradphi-int(K)-1}. It follows that $\phi_v$ is non increasing on $[0,t_v)$ for a.e.~$v\in \bb S^1$. By continuity of $\phi$ on $\interior(K)$, we get that this holds for every $v\in \bb S^1$. Together with~\eqref{e:radial-continuity-phi}, this implies that $t \mapsto \phi(tw)$ is non increasing on $[0,1]$ for every $w\in \partial K$.
\end{proof}

We end this section with a regularity property for the profile of homogeneous distances satisfying BCP on $\HH$, namely, such profiles are differentiable at $0$. The proof relies on an other sufficient condition for the non-validity of BCP that was proved in~\cite{LeDonne_Rigot_Heisenberg_BCP}, see Theorem~\ref{thm:criterion-no-bcp} below, together with Theorem~\ref{thm:radially-decreasing-profile}.

\begin{theorem} \label{thm:diff-at-0} 
Assume that BCP holds in $(\HH,d)$. Then $\phi$ is differentiable at $0$  with $\nabla\phi(0) = 0$.
\end{theorem}

\begin{proof}
Since $\phi\in \Lip_{loc}(\interior(K))$, it is sufficient to prove that for every $v\in \mathbb{S}^1$, 
\begin{equation} \label{e:diff-at-0}
\lim_{t\rightarrow 0^+} \frac{\phi(tv) - \phi(0)}{t} = 0.
\end{equation}
Indeed, assume that this holds true and let $U\subset \R^2$ be an open neighbourhood of $0$ on which $\phi$ is $L$-Lipschitz for some $L>0$. Let $\varepsilon >0$. By compactness of $\mathbb{S}^1$, one can find finitely many $v_1,\dots,v_k \in \mathbb{S}^1$ so that, for every $w\in \mathbb{R}^2$, there is $1\leq j \leq k$ such that $\measuredangle (v_j,w) \leq \varepsilon$. 
Next let $w\in U$. Let $1\leq j \leq k$ be such that $\measuredangle (v_j,w) \leq \varepsilon$. We have
\begin{equation*}
\begin{split}
\frac{|\phi(w) - \phi(0)|}{\|w\|} &\leq \frac{|\phi(w) - \phi(\langle v_j,w \rangle v_j)|}{\| w \|} + \frac{|\phi(\langle v_j,w \rangle v_j) - \phi(0)|}{\|w\|} \\
&\leq L \,\frac{\|w - \langle v_j,w \rangle v_j \|}{\| w \|} + \frac{|\phi(\langle v_j,w \rangle v_j) - \phi(0)|}{|\langle v_j,w \rangle|}  \\
&\leq L \sin \varepsilon + \frac{|\phi(\langle v_j,w \rangle v_j) - \phi(0)|}{|\langle v_j,w \rangle|} .
\end{split}
\end{equation*}
By~\eqref{e:diff-at-0}, we get
\begin{equation*}
\limsup_{\|w\| \rightarrow 0} \frac{|\phi(w) - \phi(0)|}{\|w\|} \leq L \sin \varepsilon,
\end{equation*}
and since this holds for every $\varepsilon >0$, it follows that $\phi$ is differentiable at $0$ with $\nabla\phi(0) = 0$.

To prove~\eqref{e:diff-at-0}, let $v\in \mathbb{S}^1$. We know from Theorem~\ref{thm:radially-decreasing-profile} that $$\limsup_{t\rightarrow 0^+} \frac{\phi(tv) - \phi(0)}{t} \leq 0.$$
By contradiction, assume that $$\liminf_{t\rightarrow 0^+} \frac{\phi(tv) - \phi(0)}{t} < 0.$$
Then there is $c>0$ and a positive sequence $t_n^+$ decreasing to $0$ such that
$$\frac{\phi(t_n^+ v) - \phi(0)}{t_n^+} < -c$$ for every $n\geq 0$. Set $\phi_v(t):=\phi(tv)$. With no loss of generality, we can assume that $t_0^+$ is small enough so that $\phi_v$ is $L$-Lipschitz on $[-t_0^+,t_0^+]$ for some $L>0$ and we set $$t_n^- := \min \left\{\frac{1}{2},\frac{c}{3L}\right\} t_n^+ .$$ We have $$\lim_{n\rightarrow +\infty} t_n^+ + t_n^- \leq \frac{3}{2}\, \lim_{n\rightarrow +\infty}  t_n^+= 0$$
and
\begin{equation*}
\begin{split}
\frac{\phi(t_n^+ v) - \phi(-t_n^- v)}{t_n^+ + t_n^-} &\leq\frac{\phi(t_n^+ v) - \phi(0)}{t_n^+} \cdot \frac{t_n^+}{t_n^+ + t_n^-} + \frac{\phi(0) - \phi(-t_n^- v)}{t_n^+ + t_n^-} \\
&\leq - \frac{2}{3} c + L \frac{t_n^-}{t_n^+} \leq -\frac{c}{3}.
\end{split}
\end{equation*}
Since we know from Theorem~\ref{thm:radially-decreasing-profile} that $\phi_v$ is non increasing on $[0,t_v]$ where $t_v:=\max\{t>0: tv \in K\}$, we also have 
$$\{ (tv,z) \in \HH: \; t_n^+ \leq t \leq t_v,\, z> \phi(t_n^+ v)\} \subset \HH \setminus B$$ 
and we get a contradiction from Theorem~\ref{thm:criterion-no-bcp} below.
\end{proof}

\begin{theorem} \cite[Theorem~6.3]{LeDonne_Rigot_Heisenberg_BCP} \label{thm:criterion-no-bcp}
Assume that there are $a>0$, $t^+>0$, $v\in \mathbb{S}^1$, and sequences $p_n^+=(t_n^+ v, \phi(t_n^+ v))\in \partial B$, $p_n^-=(-t_n^- v,\phi(-t_n^- v))\in \partial B$ such that 
\begin{gather*}
-t_n^-<0<t_n^+ , \quad \lim_{n\rightarrow +\infty} t_n^+ + t_n^- = 0, \\
 \phi(t_n^+ v) - \phi(-t_n^- v) < -a \,(t_n^+ + t_n^-), \\
\{ (tv,z) \in \HH: \; t_n^+ \leq t \leq t^+,\, z> \phi(t_n^+ v)\} \subset \HH \setminus B.
\end{gather*}
Then BCP does not hold in $(\HH,d)$.
\end{theorem}

\subsection{More necessary conditions under regularity assumptions} \label{subsect:additional-nec-cdt} 
We show in this section that one can slightly improve the necessary conditions proved in Section~\ref{subsect:nec-cdt-general-case} when a priori additional regularity is assumed for the profile of the homogeneous distance. As in the previous section, we let $d$ denote a homogeneous distance on $\HH$ with unit ball centered at the origin $B$ and profile $\phi:K \rightarrow [0,+\infty)$ (see Proposition~\ref{prop:LDNG}).

\medskip

Let $w\in \interior (K)$, assume that $\phi$ is differentiable on a neighbourhood of $w$, and let $v\in\R^2 \setminus\{0\}$. Given a smooth curve $\gamma : I \rightarrow \interior (K)$ defined on some open interval $I\subset \R$ containing $0$ and  
 such that $\gamma(0) = w$, we set 
$q^+_\gamma(t):=(\gamma(t),\phi(\gamma(t)))$ and  \begin{equation*}
	g_{\gamma,v}(t) := \langle \grad F_+(q^+_\gamma(t)), (\tau_{q_\gamma(t)})_*(v) \rangle  = -\langle \grad\phi(\gamma(t)) , v \rangle + \frac12 \omega(\gamma(t),v)
	\end{equation*}
for $t\in I$ small enough. Similarly, we set $q^-_\gamma(t):=(-\gamma(t),-\phi(\gamma(t)))$ and  
\begin{equation*}
	h_{\gamma,v}(t) := \langle \grad F_-(q^-_\gamma(t)), (\tau_{q^-_\gamma(t)})_*(-v) \rangle  = -\langle \grad\phi(\gamma(t)) , v \rangle - \frac12 \omega(\gamma(t),v)
	\end{equation*}
for $t\in I$ small enough. See~\eqref{e:def-F} for the definition of $F_+$ and $F_-$.

\medskip

If $\grad\phi$ is differentiable at $\gamma(t)$, then $g_{\gamma,v}$ and $h_{\gamma,v}$ are differentiable at $t$ with 
\begin{align}
\label{e:g'}
g'_{\gamma,v}(t) = - \langle H\phi(\gamma(t)) v , \gamma'(t) \rangle + \frac{1}{2} \omega(\gamma'(t),v) \\
\label{e:h'}
h'_{\gamma,v}(t) = - \langle H\phi(\gamma(t)) v, \gamma'(t) \rangle - \frac{1}{2} \omega(\gamma'(t),v)
\end{align}
where $H\phi$ denotes the Hessian of $\phi$.

\medskip

We first prove that~\eqref{e:necessary-cdt-gradphi-int(K)-1} can be improved into a strict inequality at points $w\in \interior(K) \setminus \{0\}$ where $\grad\phi$ is differentiable.

\begin{proposition} \label{prop:nec-cdt-intK-phidiff} Assume that BCP holds on $(\HH,d)$. Let $w\in \interior(K) \setminus \{0\}$ be such that $\phi$ is differentiable on a neighbourhood of $w$ and $\grad\phi$ is differentiable at $w$. Then $\langle  \grad\phi(w) , w\rangle < 0$.
\end{proposition}

\begin{proof}
We know from~\eqref{e:necessary-cdt-gradphi-int(K)-1} that $\langle  \grad\phi(w),w \rangle \leq 0$. By contradiction, assume that $\langle \grad\phi(w),w\rangle = 0$. Let $v\in \R^2$ and $\gamma : I \rightarrow \interior (K)$ be a smooth curve defined on some open interval $I\subset \R$ containing $0$
such that $\gamma(0) = w$ and $\gamma'(0) = v$. We have $\pi(q^+_\gamma(t)) = \gamma(t)\not=0$ for $t\in I$ small enough and $\measuredangle (w,\gamma(t)) \rightarrow 0$ when $t\rightarrow 0$, hence it follows from Proposition~\ref{prop:necessary-condition-bcp} that $g_{\gamma,w}(t) \geq 0$ for $t\in I$ small enough. Since $g_{\gamma,w}(0) =-\langle \grad\phi(\gamma(0)) , w \rangle +  \omega(\gamma(0),w)/2= -\langle \grad\phi(w) , w \rangle +  \omega(w,w)/2 =0$ and $g_{\gamma,w}$ is differentiable at $0$, it follows that $g'_{\gamma,w}(0)= 0$ where (see~\eqref{e:g'})
\begin{equation*}
\begin{split}
g'_{\gamma,w}(0) &= -\langle H\phi(\gamma(0))w,\gamma'(0)\rangle +\frac{1}{2}\omega(\gamma'(0),w)\\
&= -\langle H\phi(w)w,v\rangle + \frac{1}{2}\omega(v,w)~.
\end{split}
\end{equation*}
Similarly, $h_{\gamma,w}(t) \geq 0$ for $t\in I$ small enough with $h_{\gamma,w}(0) = 0$ and $h_{\gamma,w}$ differentiable at $0$. Hence it follows that $h'_{\gamma,w}(0)= 0$ where (see~\eqref{e:h'})
\begin{equation*}
h'_{\gamma,w}(0) =-\langle H\phi(w)w,v\rangle  - \frac{1}{2}\omega(v,w)~.
\end{equation*}
It follows that $\omega(v,w)=0$ for all $v\in \R^2$ and we get a contradiction since $w\not=0$.
\end{proof}

Next, we improve Theorem~\ref{thm:radially-decreasing-profile} and show that the profile has a strict maximum at the origin if $\grad \phi$ is assumed to be differentiable on $\interior(K) \setminus \{0\}$.

\begin{proposition}
Assume that BCP holds on $(\HH,d)$. Assume that $\phi \in C^1(\interior(K) \setminus \{0\})$ with $\grad \phi$ differentiable on $\interior(K) \setminus \{0\}$. Then $ \phi(w) <\phi(0)$ for all $w\in K \setminus\{0\}$.
\end{proposition}

\begin{proof} Let $w \in \interior(K) \setminus \{0\}$. Then $t\mapsto \phi(tw) \in C([0,1]) \cap C^1((0,1))$ with $$\frac{d}{dt} \phi(tw) = \langle \grad\phi(tw) , w \rangle  = t^{-1} \langle \grad\phi(tw) , tw \rangle  <0$$ for all $t\in (0,1)$ by Proposition~\ref{prop:nec-cdt-intK-phidiff}. Hence $t\in (0,1) \mapsto \phi(tw) $ is decreasing for every $w \in \interior(K) \setminus \{0\}$, which, together with~\eqref{e:radial-continuity-phi}, gives the conclusion.
\end{proof}

\begin{remark} \label{rk:nobcp-non-cont-profile} The previous proposition can be applied to prove that the homogeneous distance whose profile $\phi_1: \bb D \rightarrow [1/4,+\infty)$ is given in Example~\ref{ex:non-cont-profile} does not satisfy BCP. Indeed, in this example, one has $\phi_1 \in C^\infty(\interior(\bb D))$ with $\phi_1(0) = \phi_1(x,0)$ for every $x$ such that $|x|\leq 1$.
\end{remark} 

We now turn to the study of the properties of the Hessian of the profile at the origin. First, we prove that if $\phi$ is assumed to be differentiable on a neighbourhood of $0$ with $\grad\phi$ differentiable at $0$, then $H\phi(0)$ is semi-definite negative.

\begin{proposition} \label{prop:nec-cdt-Hphi-1}
Assume that BCP holds on $(\HH,d)$. Assume that $\phi$ is differentiable on a neighborhood of $0$ with $\grad\phi$ differentiable at $0$. Then the Hessian $H\phi(0)$ of $\phi$ at 0 is semi-definite negative.
\end{proposition}

\begin{proof}
We know from Theorem~\ref{thm:diff-at-0} that $\grad\phi(0)=0$. Let $v\in\R^2 \setminus\{0\}$ and $\gamma: I \rightarrow \interior (K)$ be a smooth curve defined on some open interval $I\subset \R$ containing $0$ 
 such that $\gamma(0) = 0$ and $\gamma'(0) = v$. We have $\pi(q^+_\gamma(t)) = \gamma(t)\not=0$ for $t\in I\setminus\{0\}$ small enough and $\measuredangle (v,\gamma(t)) \rightarrow 0$ when $t\downarrow 0^+$, hence it follows from Proposition~\ref{prop:necessary-condition-bcp} that $g_{\gamma,v}(t)\geq 0$ for $t>0$ small enough. On the other hand, we have $g_{\gamma,v}(0)= -\langle \grad\phi(0) , v \rangle +  \omega(v,v)/2 = 0$. Since $g_{\gamma,v}$ differentiable at $0$, it follows that $g_{\gamma,v}'(0) \geq 0$ where (see~\eqref{e:g'})
\begin{equation*}
g'_{\gamma,v}(0) = -\langle H\phi(\gamma(0))v,\gamma'(0)\rangle = -\langle H\phi(0)v,v\rangle~. 
\end{equation*}
Hence $\langle H\phi(0)v,v\rangle \leq 0$ for all $v\in\R^2\setminus\{0\}$, that is, $H\phi(0)$ is semi-definite negative.
\end{proof}

Finally, we improve the previous proposition and show that $H\phi(0)$ is definite negative if $H\phi$ is assumed to be differentiable at $0$.

\begin{proposition} \label{prop:nec-cdt-Hphi-2}
Assume that BCP holds on $(\HH,d)$. Assume that $\phi$ and $\grad\phi$ are differentiable on a neighborhood of $0$ with $H\phi$ differentiable at $0$. Then $H\phi(0)$ is definite negative.
\end{proposition}

\begin{proof}
We know from Proposition~\ref{prop:nec-cdt-Hphi-1} that $H\phi(0)$ is semi-definite negative. By contradiction assume that one can find $v\in\R^2 \setminus\{0\}$ such that $H\phi(0)v=0$. Arguing as in the proof of Proposition~\ref{prop:nec-cdt-Hphi-1}, let $w\in \R^2$ and let $\gamma: I \rightarrow \interior (K)$ be a smooth curve defined on some open interval $I\subset \R$ containing $0$ 
such that $\gamma(0) = 0$ and $\gamma'(0) = v$ and $\gamma''(0) = w$. We have $g_{\gamma,v}(0)=0$, $g'_{\gamma,v}(0)=0$ and $g_{\gamma,v}(t)\geq 0$ for all $t>0$ small enough with $g'_{\gamma,v}$ differentiable at $0$. This implies that $g''_{\gamma,v}(0)\geq 0$ with
\begin{equation*}
\begin{split}
g''_{\gamma,v}(0)& = -\langle \nabla \langle H\phi\, (v), v \rangle (0) , v \rangle - \langle H\phi(0)v,\gamma''(0) \rangle + \frac{1}{2} \omega(\gamma''(0),v)\\
& =  - \langle \nabla \langle H\phi\, (v), v \rangle (0) , v \rangle + \frac{1}{2} \omega(w,v).
\end{split}
\end{equation*}
Setting $C:= 2\langle \nabla \langle H\phi\, (v), v \rangle (0) , v \rangle$, it follows that 
\begin{equation*}
\omega(w,v) \geq C
\end{equation*}
for all $w\in \R^2$, and we get a contradiction since $v\not=0$.
\end{proof}


\end{document}